\documentclass{amsart}
\usepackage{amssymb}
\usepackage{amsbsy, amsthm, amsmath,amstext, amsopn}
\usepackage[all]{xy}
\usepackage{amsfonts}
\usepackage{amscd}
\newtheorem{thm}{Theorem} [section]
\newtheorem{lemma}[thm]{Lemma}

\newtheorem{corollary}[thm]{Corollary}
\newtheorem{prop}[thm]{Proposition}

\theoremstyle{definition}

\newtheorem{defn}[thm]{Definition}
\newtheorem{conj}[thm]{Conjecture}

\newtheorem{claim}[thm]{Claim}

\theoremstyle{remark}

\newtheorem{remark}[thm]{Remark}

\newtheorem{observation}[thm]{Observation}

\begin{document}

\numberwithin{equation}{section}

\newcommand{\hs}{\mbox{\hspace{.4em}}}
\newcommand{\ds}{\displaystyle}
\newcommand{\bd}{\begin{displaymath}}
\newcommand{\ed}{\end{displaymath}}
\newcommand{\bcd}{\begin{CD}}
\newcommand{\ecd}{\end{CD}}

\newcommand{\on}{\operatorname}
\newcommand{\proj}{\operatorname{Proj}}
\newcommand{\bproj}{\underline{\operatorname{Proj}}}

\newcommand{\spec}{\operatorname{Spec}}
\newcommand{\Spec}{\operatorname{Spec}}
\newcommand{\bspec}{\underline{\operatorname{Spec}}}
\newcommand{\pline}{{\mathbf P} ^1}
\newcommand{\aline}{{\mathbf A} ^1}
\newcommand{\pplane}{{\mathbf P}^2}
\newcommand{\aplane}{{\mathbf A}^2}
\newcommand{\coker}{{\operatorname{coker}}}
\newcommand{\ldb}{[[}
\newcommand{\rdb}{]]}

\newcommand{\Sym}{\operatorname{Sym}^{\bullet}}
\newcommand{\Symp}{\operatorname{Sym}}
\newcommand{\Pic}{\bf{Pic}}
\newcommand{\AAut}{\operatorname{Aut}}
\newcommand{\PAut}{\operatorname{PAut}}

\newcommand{\too}{\twoheadrightarrow}
\newcommand{\C}{{\mathbf C}}
\newcommand{\Z}{{\mathbf Z}}
\newcommand{\Q}{{\mathbf Q}}
\newcommand{\R}{{\mathbf R}}
\newcommand{\Cx}{{\mathbf C}^{\times}}
\newcommand{\Cbar}{\overline{\C}}
\newcommand{\Cxbar}{\overline{\Cx}}
\newcommand{\cA}{{\mathcal A}}
\newcommand{\cS}{{\mathcal S}}
\newcommand{\cV}{{\mathcal V}}
\newcommand{\cM}{{\mathcal M}}
\newcommand{\bA}{{\mathbf A}}
\newcommand{\cB}{{\mathcal B}}
\newcommand{\cC}{{\mathcal C}}
\newcommand{\cD}{{\mathcal D}}
\newcommand{\D}{{\mathcal D}}
\newcommand{\cs}{{\mathbf C} ^*}
\newcommand{\boldc}{{\mathbf C}}
\newcommand{\cE}{{\mathcal E}}
\newcommand{\cF}{{\mathcal F}}
\newcommand{\bF}{{\mathbf F}}
\newcommand{\cG}{{\mathcal G}}
\newcommand{\G}{{\mathbb G}}
\newcommand{\cH}{{\mathcal H}}
\newcommand{\cJ}{{\mathcal J}}
\newcommand{\cK}{{\mathcal K}}
\newcommand{\cL}{{\mathcal L}}
\newcommand{\baL}{{\overline{\mathcal L}}}
\newcommand{\M}{{\mathcal M}}
\newcommand{\Mf}{{\mathfrak M}}
\newcommand{\bM}{{\mathbf M}}
\newcommand{\bm}{{\mathbf m}}
\newcommand{\cN}{{\mathcal N}}
\newcommand{\theo}{\mathcal{O}}
\newcommand{\cP}{{\mathcal P}}
\newcommand{\cR}{{\mathcal R}}
\newcommand{\Pp}{{\mathbb P}}
\newcommand{\boldp}{{\mathbf P}}
\newcommand{\boldq}{{\mathbf Q}}
\newcommand{\bbL}{{\mathbf L}}
\newcommand{\cQ}{{\mathcal Q}}
\newcommand{\cO}{{\mathcal O}}
\newcommand{\Oo}{{\mathcal O}}
\newcommand{\OX}{{\Oo_X}}
\newcommand{\OY}{{\Oo_Y}}
\newcommand{\otY}{{\underset{\OY}{\ot}}}
\newcommand{\otX}{{\underset{\OX}{\ot}}}
\newcommand{\cU}{{\mathcal U}}\newcommand{\cX}{{\mathcal X}}
\newcommand{\cW}{{\mathcal W}}
\newcommand{\boldz}{{\mathbf Z}}
\newcommand{\qgr}{\operatorname{q-gr}}
\newcommand{\gr}{\operatorname{gr}}
\newcommand{\rk}{\operatorname{rk}}
\newcommand{\Sh}{\operatorname{Sh}}
\newcommand{\SH}{{\underline{\operatorname{Sh}}}}
\newcommand{\End}{\operatorname{End}}
\newcommand{\uEnd}{\underline{\operatorname{End}}}
\newcommand{\Hom}{\operatorname{Hom}}
\newcommand{\uHom}{\underline{\operatorname{Hom}}}
\newcommand{\uHomY}{\uHom_{\OY}}
\newcommand{\uHomX}{\uHom_{\OX}}
\newcommand{\Ext}{\operatorname{Ext}}
\newcommand{\bExt}{\operatorname{\bf{Ext}}}
\newcommand{\Tor}{\operatorname{Tor}}

\newcommand{\inv}{^{-1}}
\newcommand{\airtilde}{\widetilde{\hspace{.5em}}}
\newcommand{\airhat}{\widehat{\hspace{.5em}}}
\newcommand{\nt}{^{\circ}}
\newcommand{\del}{\partial}

\newcommand{\supp}{\operatorname{supp}}
\newcommand{\GK}{\operatorname{GK-dim}}
\newcommand{\hd}{\operatorname{hd}}
\newcommand{\id}{\operatorname{id}}
\newcommand{\res}{\operatorname{res}}
\newcommand{\lrar}{\leadsto}
\newcommand{\im}{\operatorname{Im}}
\newcommand{\HH}{\operatorname{H}}
\newcommand{\TF}{\operatorname{TF}}
\newcommand{\Bun}{\operatorname{Bun}}

\newcommand{\F}{\mathcal{F}}
\newcommand{\Ff}{\mathbb{F}}
\newcommand{\nthord}{^{(n)}}
\newcommand{\Gr}{{\mathfrak{Gr}}}

\newcommand{\Fr}{\operatorname{Fr}}
\newcommand{\GL}{\operatorname{GL}}
\newcommand{\gl}{\mathfrak{gl}}
\newcommand{\SL}{\operatorname{SL}}
\newcommand{\ff}{\footnote}
\newcommand{\ot}{\otimes}
\def\Ext{\operatorname {Ext}}
\def\Hom{\operatorname {Hom}}
\def\Ind{\operatorname {Ind}}
\def\bbZ{{\mathbb Z}}

\newcommand{\nc}{\newcommand}
\nc{\ol}{\overline} \nc{\cont}{\on{cont}} \nc{\rmod}{\on{mod}}
\nc{\Mtil}{\widetilde{M}} \nc{\wb}{\overline} \nc{\wt}{\widetilde}
\nc{\wh}{\widehat} \nc{\sm}{\setminus} \nc{\mc}{\mathcal}
\nc{\mbb}{\mathbb}  \nc{\K}{{\mc K}} \nc{\Kx}{{\mc K}^{\times}}
\nc{\Ox}{{\mc O}^{\times}} \nc{\unit}{{\bf \on{unit}}}
\nc{\boxt}{\boxtimes} \nc{\xarr}{\stackrel{\rightarrow}{x}}

\nc{\Ga}{\G_a}
 \nc{\PGL}{{\on{PGL}}}
 \nc{\PU}{{\on{PU}}}

\nc{\h}{{\mathfrak h}} \nc{\kk}{{\mathfrak k}}
 \nc{\Gm}{\G_m}
\nc{\Gabar}{\wb{\G}_a} \nc{\Gmbar}{\wb{\G}_m} \nc{\Gv}{G^\vee}
\nc{\Tv}{T^\vee} \nc{\Bv}{B^\vee} \nc{\g}{{\mathfrak g}}
\nc{\gv}{{\mathfrak g}^\vee} \nc{\RGv}{\on{Rep}\Gv}
\nc{\RTv}{\on{Rep}T^\vee}
 \nc{\Flv}{{\mathcal B}^\vee}
 \nc{\TFlv}{T^*\Flv}
 \nc{\Fl}{{\mathfrak Fl}}
\nc{\RR}{{\mathcal R}} \nc{\Nv}{{\mathcal{N}}^\vee}
\nc{\St}{{\mathcal St}} \nc{\ST}{{\underline{\mathcal St}}}
\nc{\Hec}{{\bf{\mathcal H}}} \nc{\Hecblock}{{\bf{\mathcal
H_{\alpha,\beta}}}} \nc{\dualHec}{{\bf{\mathcal H^\vee}}}
\nc{\dualHecblock}{{\bf{\mathcal H^\vee_{\alpha,\beta}}}}
\newcommand{\ramBun}{{\bf{Bun}}}
\newcommand{\ramBuno}{\ramBun^{\circ}}

\nc{\Buntheta}{{\bf Bun}_{\theta}} \nc{\Bunthetao}{{\bf
Bun}_{\theta}^{\circ}} \nc{\BunGR}{{\bf Bun}_{G_\R}}
\nc{\BunGRo}{{\bf Bun}_{G_\R}^{\circ}}
\nc{\HC}{{\mathcal{HC}}}
\nc{\risom}{\stackrel{\sim}{\to}} \nc{\Hv}{{H^\vee}}
\nc{\bS}{{\mathbf S}}
\def\Rep{\operatorname {Rep}}
\def\Conn{\operatorname {Conn}}

\nc{\Vect}{{\operatorname{Vect}}}
\nc{\Hecke}{{\operatorname{Hecke}}}

\newcommand{\ZZ}{{Z_{\bullet}}}
\nc{\HZ}{{\mc H}\ZZ} \nc{\eps}{\epsilon}

\nc{\CN}{\mathcal N} \nc{\BA}{\mathbb A}

\nc{\ul}{\underline}

\nc{\bn}{\mathbf n} \nc{\Sets}{{\on{Sets}}} \nc{\Top}{{\on{Top}}}
\nc{\IntHom}{{\mathcal Hom}}

\nc{\Simp}{{\mathbf \Delta}} \nc{\Simpop}{{\mathbf\Delta^\circ}}

\nc{\Cyc}{{\mathbf \Lambda}} \nc{\Cycop}{{\mathbf\Lambda^\circ}}

\nc{\Mon}{{\mathbf \Lambda^{mon}}}
\nc{\Monop}{{(\mathbf\Lambda^{mon})\circ}}

\nc{\Aff}{{\on{Aff}}} \nc{\Sch}{{\on{Sch}}}

\nc{\bul}{\bullet}
\nc{\module}{{\operatorname{mod}}}

\nc{\dstack}{{\mathcal D}}

\nc{\BL}{{\mathbb L}}

\nc{\BD}{{\mathbb D}}

\nc{\BR}{{\mathbb R}}

\nc{\BT}{{\mathbb T}}

\nc{\SCA}{{\mc{SCA}}}

\nc{\lotimes}{{\otimes}^{\mathbf L}}

\newcommand{\Coh}{L_{\operatorname{coh}}}

\nc{\QCoh}{L_{\on{qcoh}}}
\nc{\Perf}{L_{\on{perf}}}\nc{\Cat}{{\on{Cat}}}
\nc{\dgCat}{{\on{dgCat}}}
\newcommand{\Aut}{{\operatorname{Aut}}}
\nc{\bLa}{{\mathbf \Lambda}}

\nc{\RHom}{\mathbf{R}\hspace{-0.15em}\on{Hom}}
\nc{\REnd}{\mathbf{R}\hspace{-0.15em}\on{End}}

\title{Loop Spaces and Langlands Parameters}

\author{David Ben-Zvi}
\address{Department of Mathematics\\University of Texas\\Austin, TX 78712-0257}
\email{benzvi@math.utexas.edu}
\author{David Nadler}
\address{Department of Mathematics\\Northwestern University\\Evanston, IL 60208-2370}
\email{nadler@math.northwestern.edu}

\begin{abstract}
We apply the technique of $S^1$-equivariant localization to sheaves
on loop spaces in derived algebraic geometry, and obtain a
fundamental link between two families of categories at the heart of
geometric representation theory. Namely, we categorify the well
known relationship between free loop spaces, cyclic homology and de
Rham cohomology to recover the category of $\D$-modules on a smooth
stack $X$ as a localization of the category of $S^1$-equivariant
coherent sheaves on its loop space $\cL X$. The main observation is
that this procedure connects categories of equivariant $\D$-modules
on flag varieties with categories of equivariant coherent sheaves on
the Steinberg variety and its relatives. This provides a direct
connection between the geometry of finite and affine Hecke algebras
and braid groups, and a uniform geometric construction of all of the
categorical parameters for representations of real and complex
reductive groups. This paper forms the first step in a project to
apply the geometric Langlands program to the complex and real local
Langlands programs, which we describe.
\end{abstract}

\maketitle

\tableofcontents

\section{Introduction}

\nc{\fg}{\mathfrak g}

\nc{\Map}{\on{Map}} \nc{\fX}{\mathfrak X}

In this paper, we apply the technique of $S^1$-equivariant
localization to sheaves on loop spaces in derived algebraic
geometry. Namely, we categorify the well known relation between free
loop spaces, cyclic homology and de Rham cohomology into a theorem
that recovers the category of $\D$-modules on a smooth stack $X$ as
a localization of the category of $S^1$-equivariant coherent sheaves
on its loop space $\cL X$. This provides a fundamental link between
two families of categories at the heart of geometric representation
theory.

On the one hand, categories of equivariant $\D$-modules on flag
varieties are central in the representation theory of reductive
groups. Consider a complex reductive group $G$ with Lie algebra
$\fg$, Borel subgroup $B\subset G$, and flag variety $\cB= G/B$. The
localization theory of Beilinson-Bernstein identifies
representations of $\fg$ with global sections of (twisted)
$\D$-modules on $\cB$. In particular, highest weight representations
are realized by $B$-equivariant $\D$-modules on $\cB$, or in other
words, by $\D$-modules on the quotient stack $B\backslash \cB$.
Given a real form $G_\R$ of $G$ with associated symmetric subgroup
$K\subset G$, infinitesimal equivalence classes of admissable
representations of $G_\R$ correspond to Harish-Chandra modules for
the pair $(\fg, K)$. These in turn
are realized by $K$-equivariant $\D$-modules on $\cB$, or in other
words, by $\D$-modules on the quotient stack $K\backslash \cB$. By
the work of Adams, Barbasch and Vogan \cite{ABV} and Soergel
\cite{Soergel} on the real local Langlands correspondence,
(untwisted) $\D$-modules on $K\backslash \cB$ also appear as the
Langlands parameters for representations of real forms of the
{Langlands dual group} $G^\vee$. An important unifying structure in
the study of all these cases is the natural intertwiner or
convolution action of the finite Hecke algebra and braid group
carried by these categories.

On the other hand, categories of equivariant coherent sheaves on the
Springer resolution, the Steinberg variety, and their relatives have
played a prominent role in geometric representation theory since the
work of Kazhdan-Lusztig on the $p$-adic local Langlands conjecture.
These categories often arise as geometric versions of
representations of {affine} Hecke algebras and braid groups. For
example, a fundamental theorem of Kazhdan-Lusztig indentifies the
affine Hecke algebra of the Langlands dual group $G^\vee$ with the
Grothendieck group of equivariant coherent sheaves on the Steinberg
variety of $G$. Recent developments in the tamely ramified geometric
Langlands program (in particular, work of Bezrukavnikov and
collaborators, see \cite{Roma ICM}, and Frenkel-Gaitsgory \cite{FG})
have significantly advanced our understanding of the underlying
categorical structure. In particular, there are intimate connections
with representations of quantum groups, modular representations of
Lie algebras, and critical level representations of the loop algebra
for the Langlands dual group $G^\vee$.

The main point of this paper is to connect these two families of
categories by a categorified, algebro-geometric version of
$S^1$-equivariant localization. Namely, we explain that the Springer
resolution and Steinberg variety
are the derived loop spaces of the quotient stacks $G\backslash \cB$
and $B\backslash \cB$ respectively. Motivated by the representation
theory of real groups we further introduce a collection of spaces,
the {\em Langlands parameter spaces}, which are similarly related to
the quotients $K\backslash \cB$ by symmetric subgroups. To make all
of this precise, we use the formalism of derived algebraic geometry
in which it makes sense to discuss the quotient stacks of algebraic
geometry and the loop spaces of algebraic topology at the same
time. Our main application of this viewpoint is a categorical form of
quantization of cotangent bundles which is particularly well adapted
to representation theory. Namely, the derived category of
quasicoherent sheaves on the loop space of a smooth stack carries a
circle action, and the resulting equivariant derived category is
expressed in terms of $\D$-modules on $X$.\footnote{It is possible to transport to the setting of derived
loop spaces some of the varied structures on free loop spaces in
topology. For example, given a stack $X$, quasicoherent sheaves on
its derived loop space $\cL X$ form a braided tensor category, a
derived form of the {\em Drinfeld double} of quasicoherent sheaves
on $X$, which is part of a collection of topological field theory
operations. We plan to elaborate on this structure in the future.}

In the special case when $X$ is one of the quotient stacks
$G\backslash \cB$, $B\backslash \cB$, or $K\backslash \cB$, we
recover the derived category of coherent $\D$-modules on $X$ from
coherent sheaves on the Springer resolution, Steinberg variety or
{Langlands parameter space} respectively. This provides a direct
connection between finite and affine geometric representation theory
of Hecke algebras and braid groups. In the case of the Langlands
parameter space, we obtain a uniform 
geometric construction of all of the categorical parameters for
representations of real reductive groups. This is a crucial first
step in a project to apply ideas from the {geometric} Langlands
program to the complex and real local Langlands programs.

\medskip

The paper is organized as follows. In Section \ref{overview}, we
provide a detailed overview of the results of this paper, and in
Section \ref{applications}, an introduction to the intended
applications and our general project. In Section \ref{loop section},
we discuss properties of loop spaces and Hochschild (or small loop)
spaces in derived algebraic geometry. (The Appendix provides a quick
introduction to the theory of derived stacks.) In Section
\ref{sheaves section}, we study equivariant sheaves on derived
stacks. Our aim is to explain a categorification of the relation
between $S^1$-equivariant homology of loop spaces (cyclic homology)
and de Rham cohomology. Then in Sections \ref{Steinberg section} and
\ref{Langlands section}, we recall the role of the flag variety and
Steinberg variety in representation theory, introduce the Langlands
parameter spaces, and explain the resulting new perspective on the
parametrization of Harish-Chandra modules for real groups.

\medskip

The discussion in Section \ref{applications} summarizes a series of
papers in preparation \cite{geometric Vogan, complex,base change} on
the representation theory of real reductive groups. Among the
results are a conceptual geometric proof of (an affine
generalization of) Vogan duality, a canonical equivalence between
Harish-Chandra bimodules for Langlands dual complex groups
(compatible with Hecke actions), and a derivation of a strengthened
form of Soergel's conjecture from a geometric version of the
principle of automorphic base change.

\subsection{Acknowledgments}
We have benefited tremendously from many conversations on the
subjects of loop spaces, cyclic homology and homotopy theory. We
would like to thank Kevin Costello, Ian Grojnowski, Mike Mandell,
Tom Nevins, Tony Pantev, Brooke Shipley, Dima Tamarkin, Constantin
Teleman, Boris Tsygan, Amnon Yekutieli, and Eric Zaslow for their
many useful comments. We would like to especially thank Andrew
Blumberg for his explanations of homotopical algebra, Roman
Bezrukavnikov for his continued interest, many helpful discussions
and for sharing with us unpublished work, and finally Jacob Lurie
and Bertrand To\"en for introducing us to the beautiful world of
derived algebraic geometry and patiently answering our many
questions.

The first author was partially supported by NSF grant DMS-0449830,
and the second author by NSF grant DMS-0600909. Both authors would
also like to acknowledge the support of DARPA grant
HR0011-04-1-0031, and to thank Edward Frenkel and Kari Vilonen for
their continued support and encouragement.


\section{Overview}\label{overview}
\subsection{Derived loop spaces.}
The free loop space $LX$ of a topological space $X$ comes equipped
with many fascinating structures. Of fundamental importance for our
purposes is the fact that $LX$ carries a circle action by loop
rotation with the constant loops $X$ as fixed points. The theory of
equivariant localization for circle actions, relating topology of an
$S^1$-space and its fixed points, has been applied to spectacular effect
in this setting in the work of Witten and many others.

\medskip

The homology and $S^1$-equivariant homology of $LX$ are intimately
related to the Hochschild homology and cyclic homology of cochains
on $X$ respectively. In general, Hochschild homology and cyclic
homology provide a purely algebraic or categorical approach to the
geometry of free loop spaces. When applied to commutative rings or
schemes, Hochschild homology captures the algebra of differential
forms, while cyclic homology captures de Rham cohomology. In other
words, cyclic homology allows us to view the de Rham complex of a
scheme as an algebraic analogue of $S^1$-equivariant cochains on a
free loop space.

The above interpretation of de Rham cohomology is quite familiar in
mathematical physics, or more specifically, in supergeometry.
Namely, consider the cohomology of the circle $H^*(S^1,\C)=\C[\eta]$
(with $\eta$ of degree one) as the supercommutative ring of
functions on the odd line $\C^{0|1}$. For a smooth scheme $X$
(likewise for smooth manifolds), the mapping space
$\Map(\C^{0|1},X)$ is the superscheme given by the odd tangent
bundle $\BT_X[-1]$ (we will remember its $\Z$-grading of one). Thus
one may think of $\BT_X[-1]$ as a linearized analogue of the free
loop space. Observe that functions on $\BT_X[-1]$ are simply
differential forms $\Omega_X^{-\bullet}=\Sym \Omega_X[1]$
placed in negative (homological) degrees.
The analogue of the $S^1$-action on the loop space is translation
along $\C^{0|1}$. 
This action defines a canonical,
square zero, odd vector field on $\BT_X[-1]$ which is easily seen to
be the de Rham differential considered as an odd derivation of
$\Omega_X^{-\bullet}$.

Another point of view on the same construction is to model the
circle by two points connected by two line segments, and to
interpret concretely what maps from such an object to $X$ should be.
First, mapping two points to $X$ defines the product $X\times X$.
One of the line segments connecting the points says the points are
equal: we should impose the equation $x=y$ that defines the
diagonal $\Delta\subset X\times X$. Then the other line segment says
the points are equal again, so we should impose the equation $x=y$
again, or in other words, take the self-intersection
$\Delta\cap\Delta\subset X\times X$. Of course, we could interpret
this naively as being a copy of $X$ again, however the intersection
is far from transverse. For $\cO_X$ the ring of functions on $X$,
the tensor product $\Oo_X\otimes_{\Oo_{X\times X}} \Oo_X$ needs to
be derived (there are higher Tor terms). If we use the Koszul
complex to calculate the derived tensor product
$\Oo_X\otimes^{\BL}_{\Oo_{X\times X}}\Oo_X$, we again find the
commutative differential graded ring $\Omega_X^{-\bullet}$. Thus if
we accept the broader framework of supergeometry as a correction for
the degenerate intersection, then we again discover $\BT_X[-1]$ as a
model for the loop space of $X$.

\medskip

We would like to apply a version of the above picture in more
complicated contexts where the geometry of a smooth variety $X$ is
replaced by the equivariant geometry of $X$ with respect to an
algebraic group action. Our motivating examples are flag varieties
equipped with the action of various groups of significance in
representation theory. Thus we would like to work in a context that
generalizes schemes simultaneously in two directions. We need to be
able to perform the following operations without fear of missing
some aspect of the geometry:

\begin{enumerate}
\item[$\bullet$] Quotients (and more general gluings or colimits) of schemes
\item[$\bullet$] Intersections (and more general fiber products or limits) of schemes
\end{enumerate}

In both cases, passing to derived versions of the above operations
guarantees that no information is lost. The need to correct
quotients has long been recognized in algebraic geometry, and a
powerful and flexible solution is provided by the theory of stacks.
Within the framework of this theory, there is an interesting formal
substitute for the loop space of a stack $X$. Namely, we can
consider the classifying stack $pt/\Z=B\Z$ as a version of the
circle, and define the {\em inertia space} of $X$ as the mapping
stack
$$
IX=\Hom(B\Z,X).
$$
Since the circle $B\Z$ is purely homotopical, it cannot map
nontrivially along the scheme direction of $X$ but only in the
stacky direction. As a result, the objects of $IX$ are objects of
$X$ equipped with automorphisms. In particular for the stack
$BG=pt/G$, we find that $I(BG)$ is the adjoint quotient $G/G$. Note
that the inertia stack $IX$ of a scheme $X$ is nothing more than $X$
again, so we have not yet accounted for the excess self-intersection
of the diagonal.

The need to correct for degenerate intersections has come to
prominence more recently, in particular through the derived moduli
spaces vision of Drinfeld, Deligne, Feigin and Kontsevich \cite{Kont
dg}, and in particular the theory of virtual fundamental classes
(see \cite{BeFa,Behrend,CF-K} and references therein). Derived
intersections and fiber products are also necessary tools in
geometric representation theory since categories of sheaves are
sensitive to derived structures. In recent years, through the work
of To\"en-Vezzosi, Lurie and others, the foundations of a theory of
{\em derived schemes and stacks} have emerged, elegantly combining
algebraic geometry and homotopical algebra. In broad outline, one
replaces commutative rings by simplicial commutative rings (or
connective commutative differential graded rings, in characteristic
zero), and functors of points take values not in sets but in
topological spaces. (One of the formal but complicated aspects of
the story is that both the domain and target of such functors must
be treated with the correct homotopical understanding.) For the
reader's convenience, we provide an overview of the theory of
derived stacks in the Appendix, and we recommend To\"en's survey
\cite{Toen} for more details and references.

The framework of derived stacks allows one to correctly take
quotients (since functors of points take values in topological
spaces), and to correctly form intersections (since functors of
points are defined on simplicial commutative rings). One can
immediately import all of the structures of homotopy theory to this
setting. In particular, we can look at the derived stack of maps
from $S^1=B\Z$ to another stack $X$. This gives an enhanced version
of the inertia stack which we simply call the {\em loop space}
$$
\cL X=\R \Hom(S^1,X).
$$
It combines the notion of odd tangent bundle and inertia stack: $\cL
X=\BT_X[-1]$, for $X$ a smooth scheme, and $\cL (BG)=G/G$ for the classifying space
of an algebraic group $G$.
In general, for $X$ a smooth scheme equipped with a $G$-action, the
loop space $\cL (X/G)$ combines even directions coming from
stabilizers of orbits in $X$, and odd directions coming from normals
to the orbits.

Inside of the loop space $\cL X$, we single out the
{small loop space} or {Hochschild space} $\cH X$. By definition, it
is the formal completion of $\cL X$ along the constant loops
$X\subset \cL X$. For a smooth stack, we show that $\cH X$ is the formal
completion of the shifted tangent bundle $\BT_X[-1]$ along its zero section. In particular,
$\cH(BG)=\widehat{G}/G$ is the adjoint quotient of the formal group,
while $\cH X=\cL X$ for $X$ a smooth scheme. The key motivation
for introducing $\cH X$ is that
small loops are local objects on $X$ in a suitable sense.

\medskip

Although not needed in this paper, it is interesting to revisit
(Section \ref{structures}) some of the familiar structures on loop
spaces and Hochschild chains in the setting of a smooth scheme $X$.
Namely, $\cL X$ forms a family of groups over $X$ (in an appropriate
homotopical sense),
and $\cH X$ is the corresponding family of formal groups. Functions
on $\cH X$ are simply the Hochschild chains of $X$, and the
Hochschild-Kostant-Rosenberg theorem then plays the role of the
Baker-Campbell-Hausdorff theorem (as explained in \cite{Mark} and
expanded on in \cite{Rama1,Rama2, RW}). It identifies the formal
completion of the Lie algebra given by the odd tangent bundle
$\BT_X[-1]$ with the formal group $\cH X$ itself. The Lie algebra
structure on $\BT_X[-1]$ is given by the Atiyah class as explained
in~\cite{Kap RW}. The Hochschild cochains are the enveloping algebra
of $\BT_X[-1]$, and hence may be thought of as distributions
supported on constant loops inside small loops. Thus Hochschild
cochains are analogues of the algebras of chiral differential
operators of Malikov, Schechtman and Vaintrob, while the composition
of loops is analogous to the factorization structure on the small
formal loop space of Kapranov and Vasserot.


\subsection{Sheaves on loop spaces.}
We are primarily
interested in the loop space $\cL X$ of a stack for its derived
category of quasicoherent sheaves (or rather its differential graded (dg)
enhancement) which we denote by $\QCoh(\cL X)$.\footnote{By general formalism,
the category $\QCoh(\cL X)$ inherits the rich structures of loop
spaces. The theory of string topology provides an appealing way to
organize many of these structures. It describes the
natural operations on the homology of loop spaces as a part of
two-dimensional topological field theory. In the setting of derived
algebraic geometry, one can show that $\QCoh(\cL X)$ possesses a
categorified version of the string topology operations carried by
the homology of loop spaces. In particular, the pair of pants
defines a braided tensor structure on $\QCoh(\cL X)$ (more
precisely, an $E_2$-category structure). This generalizes the notion
of the usual Drinfeld double of $G$ whose modules are
$\QCoh(\cL(BG))=\QCoh(G/G)$.} In order to avoid technical
complications we will work throughout with $X$ a smooth Artin
(1-)stack with affine diagonal, though much of the paper could be generalized to higher
stacks and presumably with proper care to singular and derived targets
as well.
We will focus on the dg derived
category of quasicoherent sheaves with coherent cohomology
which we denote by $\Coh(\cL X)$.

The loop space $\cL X$ automatically carries an action of the group
$S^1$, which may be expressed for example through the Connes
formalism of cyclic objects. We may consider $S^1$-equivariant
sheaves on the loop space $\cL X$, or equivalently, sheaves on the
space of unparametrized loops $\cL X/S^1$. The $S^1$-action is also
manifested as an automorphism of the identity functor of $\QCoh(\cL
X)$, which on a sheaf $\cF$ is given by the monodromy of $\cF$ along
the $S^1$-orbits. Roughly speaking, equivariant sheaves are given by
sheaves on $\cL X$ equipped with a homotopy between their monodromy
operator and the identity. On the other hand, on the space of small
loops $\cH X\simeq \wh{\BT}[-1]$ we have the odd derivation given by
the de Rham differential. To relate the two, we discuss a general
Koszul duality formalism for equivariant sheaves in Section
\ref{equivariant}. In the case of $S^1$ acting on a point, this
gives rise to the well-known equivalence between cyclic vector
spaces and complexes with exterior algebra action (mixed complexes).
After restricting to the Hochschild space $\cH X$ of small loops, we
check that an $S^1$-equivariant structure is precisely a lifting of
the odd vector field on $\BT_X[-1]$ to the sheaf. In other words,
the sheaf becomes endowed with an action of the de Rham differential
-- considered as a {\em homotopy} on the {\em homological} complex
$\Omega_X^{-\bullet}$ -- or equivalently an action of the algebra
$\Omega_X^{-\bullet}[d]$ in which we've adjoined $d$ in degree $-1$.
We thus have the following categorification of the relation between
de Rham cohomology and cyclic homology:

\begin{thm}[Theorem \ref{cyclic and D} below]
For a smooth Artin stack $X$, there is a canonical quasi-equivalence
of dg derived categories
$$
\QCoh(\cH X)^{S^1} \simeq \QCoh(X, \Omega_X^{-\bullet}[d])
$$
preserving subcategories of coherent objects.
\end{thm}

Modules over the de Rham complex are intimately related to sheaves
with flat connection, or $\D$-modules, on $X$. As explained in
\cite{Kap dR, BD Hitchin} the relation between differential
operators $\D_X$ and the de Rham complex $(\Omega_X,d)$ is a form of
Koszul duality.\footnote{Koszul duality does not naively give an
equivalence on categories of quasicoherent sheaves. This can be
corrected by modifying the notion of equivalence of de Rham modules
(as in \cite{BD Hitchin}), or by killing some large $\D$-modules
which are missed by the de Rham functor.
In either case, the categories of coherent modules, which are our
primary interest, are unaffected.} First, a form of Koszul duality
identifies $\Omega_X^{-\bullet}[d]$-modules with dg modules over the
Rees algebra $\cR_X$ of $\D_X$ placed in positive even degrees.

To recover the category of $\D_X$-modules, or equivalently dg
modules over the de Rham complex $(\Omega_X^{\bullet},d)$, we pass
to the {\em periodic} version, by tensoring the category of modules
with the ring $\C[u,u\inv]$. On the other side, this amounts to
performing the usual localization of $S^1$-equivariant cohomology.
Observe that $S^1$-equivariant sheaves form a category linear over
the ring $H^*_{S^1}(pt)=\C[u]$, with $|u|=2$. If we invert $u$, we
obtain a $\Z/2\Z$-periodic dg category
$$
\QCoh(\cH X)_{per}^{S^1}=\QCoh(\cH X)^{S^1}\ot_{\C[u]}\C[u,u\inv].
$$
The inversion of $u$ matches up with the localization from the Rees
algebra to $\D_X$ itself, resulting in the following relation
between $\D_X$-modules and the periodic cyclic category of $X$:

\begin{corollary}
For a smooth Artin stack $X$, there are canonical quasi-equivalences
of dg derived categories of coherent sheaves
$$
\Coh(\cH X)^{S^1}\simeq \Coh(X,\cR_X)
$$
$$
\Coh(\cH X)_{per}^{S^1}\simeq \Coh(X,\D_X)_{per}
$$
\end{corollary}

This picture of $\D$-modules as $S^1$-equivariant sheaves on the
Hochschild space is of course closely related to many other ways to
express flat connections. Most directly, the dg Lie algebra
$\BT_X[-1]$ acts by endomorphisms of the identity of $\QCoh(X)$.
Namely, the action $\BT_X[-1]\ot\cF\to\cF$ is given by the Atiyah
class of $\cF$, which is the one-jet extension $\cJ
\cF\in\Ext^1(\cF,\cF\ot\Omega_X)$. A trivialization of the Atiyah
class of a sheaf is precisely the data of a connection. The
structure of trivialization of the monodromy on sheaves on $\cH X$
is related by pullback to trivialization of the Atiyah class on $X$.
This gives an alternative route to recover the relation between
$S^1$-equivariant sheaves on $\cH X$ and flat connections on $X$.

For another point of view, note that Koszul duality identifies
sheaves on the odd tangent bundle (modules for
$\Omega^{-\bullet}_X$) with sheaves on the cotangent bundle (modules
for $\Sym \BT_X$). Passing from the graded ring
$\Omega^{-\bullet}_X$ to the differential graded ring
$(\Omega_X^{-\bullet},d)$ corresponds to deforming the graded ring
$\Sym T_X$ to the filtered ring $\D_X$. Our original motivation for
this story came from the observation that in applications to
representation theory, it is often easier to identify the
differential $d$ than the deformation quantization $\D_X$. Thus we
think of loop spaces and their circle action as a useful geometric
counterpart to cotangent bundles and their quantization. The same
paradigm appears in relating string topology of $X$ (that is,
topology of $LX$) to the A-model (Fukaya category) of $T^*X$. In
that sense, this picture is a counterpart to the emerging relation
between Fukaya categories and $\D$-modules or constructible sheaves
(see \cite{KW,NZ,Nad}).

We would also like to mention the beautiful work of Simpson and
Teleman \cite{ST} on de Rham's theorem on stacks. It deals with
$\D$-modules on general stacks as sheaves equivariant for the formal
neighborhood of the diagonal. This natural picture is related to our
loop space picture by a somewhat awkward shift of grading (taking
the homological or simplicial ring given by Hochschild homology and
translating it into a cohomological or cosimplicial ring giving the
usual de Rham stack). On the categorical level, this requires
lifting to a ``mixed" setting \`a la \cite{BGS}. We hope to return
to this issue in the future.


\subsection{Equivariant $\D$-modules on flag varieties}
We now turn to representation theory, which is our primary impetus
for this work and which takes up the last two sections of the paper.
Our motivation is the observation that (the equivariant versions of)
the Steinberg variety and its relatives, the Langlands parameter
spaces, are loop spaces, and that the corresponding equivariant
localization (as described above) matches well with the Langlands
program. In this section and the next we review some of the
background in geometric representation theory, with a large emphasis
on representations of real Lie groups. Our results are described in
Section \ref{results}. We recommend that the reader interested in
applications in geometry or complex groups skip the real material on
first reading (see Section \ref{Langlands spaces} for the definition
of the Langlands parameter spaces).

\medskip

Let $G$ be a complex reductive group with Lie algebra $\fg$, Borel
subgroup $B\subset G$, and flag variety $\cB= G/B$. A primary
motivation for studying $\D$-modules on algebraic stacks comes from
the localization theory of Beilinson-Bernstein. It identifies
representations of $\g$ with
global sections of twisted $\D$-modules on $\cB.$ Furthermore, given
a subgroup $K\subset G$, it identifies modules for the
Harish-Chandra pair $(\g,K)$ with global sections of $K$-equivariant twisted
$\D$-modules on $\cB$. The following well-known cases are of
traditional interest:

\medskip

(1) The case $K=G$ gives the Borel-Weil description of irreducible
algebraic (equivalently, finite-dimensional) representations of $G$
as sections of equivariant line bundles on $\cB$.

\medskip

(2) The case $K=B$  gives highest-weight modules as sections of
twisted $\D$-modules smooth along the Schubert stratification. Such
modules are closely related to the objects of category $\Oo$
of Bernstein-Gelfand-Gelfand.
Via the identification
$$
B\backslash \cB= G\backslash (\cB\times \cB),
$$
they are also closely related to Harish-Chandra bimodules, and
thus also to admissible representations
of $G$ considered as a real Lie group.

\medskip

(3) Finally, the case of a {symmetric subgroup} $K\subset G$ is of
fundamental interest for its relation to real groups. By a symmetric
subgroup, we mean the fixed points of an algebraic involution
${\eta}$. Such involutions correspond to antiholomorphic involutions
$\theta$ and hence real forms $G_\R$ of $G$ via the assignment $
\theta=\eta\circ \kappa $ where $\kappa$ is a commuting Cartan
involution of $G$.
In this case, the irreducible $K$-equivariant twisted $\D$-modules
give infinitesimal equivalence classes of irreducible admissible
representations of $G_\R$. 

\medskip

In this paper, we will restrict our attention to untwisted
$\D$-modules. This choice reflects the fact that we intend to apply
the ideas of this paper to $\D$-modules in their role as Langlands
parameters. According to Adams-Barbasch-Vogan~\cite{ABV} and
Soergel~\cite{Soergel}, it is untwisted $\D$-modules which arise in
this way. In addition, rather than abelian categories, we will work
with the corresponding $K$-equivariant derived categories
$\D_K(\cB)$ of $\D$-modules on $\cB$ for each of the above subgroups
$K$. (Our convention implicit in the notation $\D_K(\cB)$ will be to
consider only coherent $\D$-modules.) It is important to note that
the equivariant categories are {\em not} the derived categories of
the abelian categories. For example, there are higher Ext groups
between equivariant sheaves as can be seen already in the case
$K=G=\G_m$ where we have
$$
\D_K(\cB) = \D_{\G_m}(pt) = H^*_{S^1}(pt)-\module.
$$

In the case $K=B$, the category  $\D_B(\cB)$ is a monoidal category
under convolution, which acts on any category of the form
$\D_K(\cB)$.
The action gives rise to an action on $\D_K(\cB)$ of the Artin braid
group of $G$, generalizing the classical principal series
intertwining operators. The Koszul duality theorem of
Beilinson-Ginzburg-Soergel~\cite{BGS} identifies $\D_B(\cB)$
 with the derived category of Harish-Chandra bimodules
 with unipotent {generalized} infinitesimal character.

\medskip

The $K$-equivariant derived categories $\D_K(\cB)$ for a symmetric
subgroup $K\subset G$ play an important role in the local Langlands
program over the real numbers \cite{ABV,Soergel}. In \cite{ABV},
Adams, Barbasch and Vogan recast the parametrization of irreducible
admissible representations of real forms of a complex reductive
group $\Gv$ (in the form it developed from the works of
Harish-Chandra, Langlands, Shelstad and others) in terms of these
categories.\footnote{Note the nonstandard switching of the roles of
$G$ and its Langlands dual group $G^\vee$. This notational choice is
(partially) excused by the fact that this paper takes place
completely on the spectral side of Langlands duality.} They
reinterpret Vogan's character duality \cite{Vogan} as an extension
of this classification to the Grothendieck groups of such
categories.

To explain the general shape of this picture, we introduce some
further notation. Fix once and for all an algebraic involution
$\eta$ of the complex reductive group $G$. Then associated to $\eta$
is a finite collection $\Theta(\eta)$ of antiholomorphic involutions
of the Langlands dual group $\Gv$.
For each $\theta\in \Theta(\eta)$, we write
$G^\vee_{\R,\theta}\subset \Gv$ for the corresponding real form.

Let $\h$ denote the universal Cartan algebra of $\fg$, and let $W$
denote its Weyl group. For each $[\lambda]\in \h/W\simeq
(\h^{\vee})^*/W$, we write $\HC_{\theta, [\lambda]}$ for the
category of Harish-Chandra modules for the real form $\Gv_{\R,
\theta}$ with generalized infinitesimal character given by
$[\lambda]$.

For simplicity, we will restrict our attention to the case when $[\lambda]$ is regular.
Fix a semisimple lift $\lambda \in \fg$ of the infinitesimal
character $[\lambda]$, and let $\alpha\in G$ denote the element
$\exp(\lambda)$. Let $G_\alpha\subset G$ be the reductive subgroup
that centralizes $\alpha$, and let $\cB_\alpha=G_\alpha/B_{\alpha}$
be its flag variety. Consider the finite set of twisted conjugacy
classes
$$
\Sigma({\eta,\alpha})=\{\sigma\in G|\sigma\eta(\sigma)=\alpha\}/G.
$$
Each $\sigma\in \Sigma({\eta,\alpha})$ defines an involution of
$G_\alpha$, and we write $K_{\alpha, \sigma}\subset G_\alpha$ for
the corresponding symmetric subgroup.

\begin{thm}[\cite{Vogan,ABV}]
There is a perfect pairing between the Grothendieck groups
$$
\bigoplus_{\theta\in \Theta(\eta)} K({\HC}_{\theta,[\lambda]}) \quad
\leftrightarrow \quad \bigoplus_{\sigma\in \Sigma({\eta,\alpha})}
K(\D_{K_{\alpha,\sigma}}(\cB_\alpha))
$$
respecting Hecke symmetries.
\end{thm}

In \cite{ABV}, this theorem is combined with the microlocal geometry
of $\D$-modules (in particular, the geometry of the cotangent
bundles $T^*(K_{\alpha,\sigma}\backslash \cB_\alpha)$) to study
Arthur's conjectures.

\medskip

Soergel \cite{Soergel} extended this $K$-theoretic picture to the
categorical level, conjecturing the existence of a (Koszul duality)
equivalence of derived categories of Harish-Chandra modules and
derived categories of $\D$-modules on the corresponding geometric
parameter spaces $K_{\alpha,\sigma}\backslash
\cB_\alpha$. A form of Soergel's conjecture reads as follows.
To state it, we write ${\HC}^{pro}_{\theta,[\lambda]}$ for the
abelian category obtaind by pro-completing
${\HC}_{\theta,[\lambda]}$ with respect to the generalized
infinitesimal character.

\begin{conj}[\cite{Soergel}]
There is an equivalence of derived categories
$$
\bigoplus_{\theta\in \Theta(\eta)}
\D({\HC}^{pro}_{\theta,[\lambda]}) \simeq \bigoplus_{\sigma\in
\Sigma({\eta,\alpha})} \D_{K_{\alpha,\sigma}}(\cB_\alpha)
$$
\end{conj}

Soergel establishes this conjecture in the case of tori, $SL_2$ and
most importantly for complex groups $\Gv$ (considered as real forms
of their complexifications).

As mentioned in \cite{ABV}, it is important to find a way to fit
together the geometric parameter spaces $K_{\alpha,\sigma}\backslash
\cB_\alpha$ for varying $\alpha$. In particular, such a family is
necessary if one hopes to have a uniform picture for representations
with different infinitesimal characters. One of the outcomes of this
work is a solution to this problem. As we will see, the geometric
parameter spaces naturally emerge from the loop geometry of
Steinberg varieties.


\subsection{Equivariant coherent sheaves on Steinberg varieties}
Consider a complex reductive group $G$ with Borel subgroup $B\subset
G$, maximal unipotent subgroup $U\subset B$, universal Cartan
$H=B/U$, and flag variety $\cB = G/B$.

\medskip

We will use the name Grothendieck-Springer simultaneous resolution
for the smooth scheme $\wt{G}$ that parametrizes pairs $(g,B)$ of an
element $g\in G$ and a Borel subgroup $B\in \cB$ such that $g\in B$.
We will always think of $\wt G$ as a family over the universal
Cartan $H$ via the canonical projection
$$
(g,B) \mapsto [g] \in B/U.
$$
The fiber over the identity $e\in H$ is the usual Springer
simultaneous resolution $\wt{G}_e$ that parametrizes pairs $(u,B)$
of a Borel subgroup $B\in \cB$ and a unipotent element $u\in B$. Its
two canonical projections exhibit $\wt{G}_e$ on the one hand as the
cotangent bundle $T^*\cB$, and on the other hand as a smooth
resolution of the unipotent cone of $G$.

\medskip

We will use the name Steinberg variety for the scheme
$\St_{}$ that parametrizes triples $(g,B_1,B_2)$ of a pair of
Borel subgroups $B_1,B_2\in \cB$ and an element in their
intersection $g\in B_1\cap B_2$. We will always think of
$\St_{}$ as a family over the product of two copies of the
universal Cartan $H\times H$ via the canonical projection
$$
(g,B_1,B_2) \mapsto ([g]_1,[g]_2) \in B_1/U_1 \times B_2/U_2
$$
Its image consists of pairs of elements which are related by the
Weyl group action. The fiber over the identity $(e,e)\in H\times H$
is the usual Steinberg variety $\St_{e,e}$ that
parametrizes triples $(u,B_1,B_2)$ of a pair of Borel subgroups
$B_1,B_2\in \cB$ and a unipotent element in their intersection $u\in
B_1\cap B_2$. In general, the Steinberg variety
$\St_{}$ is connected, but has irreducible components  labeled
by the Weyl group, and hence as long as $G$ is nonabelian,
$\St_{}$ is singular.

\medskip

The fundamental relationship between the Grothendieck-Springer
simultaneous resolution $\wt G$ and the Steinberg variety $\St_{}$
is that the latter is given by the fiber product
$$
\St_{} = \wt{G}\times_G\wt{G}.
$$
Because the projection $\wt G\to G$ is semi-small, the derived fiber product coincides
with the above naive fiber product.
By the usual formalism of correspondences, this implies that
coherent sheaves on $\St_{}$ form a convolution algebra which
acts on coherent sheaves on $\wt G$ (see \cite{CG} for a detailed
exposition). The importance of this construction in representation
theory derives from the work of Kazhdan-Lusztig on the tamely
ramified $p$-adic local Langlands program (the Deligne-Langlands
conjecture) \cite{KL}. The starting point of this theory is their
identification of the Grothendieck group of $(G\times
\Cx)$-equivariant coherent sheaves on the fiber $\St_{e,e}$
with the affine Hecke algebra of the Langlands dual group $\Gv$.
As a result, all modules over the affine Hecke algebra admit a
spectral description in terms of various Grothendieck groups of
equivariant coherent sheaves on the Springer resolution.

More recently, the equivariant derived categories of coherent
sheaves underlying the above $K$-theoretic story have begun to be
understood by the work of Bezrukavnikov and collaborators (including
Arkhipov, Ginzburg, Mirkovic, and Rumynin; see \cite{Roma ICM} for
an overview), the work of Frenkel-Gaitsgory (\cite{FG}, see
\cite{ramifications} for an overview) and the work of Gukov-Witten
\cite{GW}. 
These advances have a wide range of applications to modular
representation theory and the Lusztig conjectures, representation
theory of quantum groups, representations of affine algebras at the
critical level, and the local geometric Langlands conjecture.
In a striking categorification of the work of Kazhdan-Lusztig,
Bezrukavnikov has identified the equivariant derived category of the
Steinberg variety (as a monoidal differential graded category) with
the affine Hecke category of $\D$-modules on the affine flag variety
of the Langlands dual group $\Gv$. One immediate consequence is that
the equivariant derived categories of Springer fibers carry actions
of the affine Hecke category, and hence of the affine braid group.
More generally, all module categories over the affine Hecke category
admit a spectral description in terms of equivariant coherent
sheaves on Springer fibers.

\subsubsection{Langlands parameter spaces}\label{Langlands spaces}
In an
ongoing project to better understand representations of real groups,
Vogan duality, and Soergel's conjecture, the fundamental objects
that arise are certain (automorphic) module categories for the
affine Hecke category. By definition, a successful characterization
of these categories would involve their spectral description as
module categories over equivariant coherent sheaves on the Steinberg
variety. Thanks to various structures on these categories, it is
possible to guess what form this spectral description should take.

As in Vogan duality and Soergel's conjecture, our starting point is
a fixed algebraic involution $\eta$ of the group $G$. In
Section~\ref{Langlands section}, given such an involution $\eta$, we
introduce a scheme $\St^\eta$ which we call the Langlands parameters
space. By construction, it parametrizes pairs consisting of a Borel subgroup $B\subset G$
and an element $g\in G$ whose $\eta$-twisted square is contained in $B$:
$$
\St^\eta=\{(g,B)\in G\times\cB \;|\; g\eta(g) \in B\}.
$$
The group $G$ naturally acts on $\St^\eta$ by twisted conjugation.
By the general formalism of correspondences, equivariant coherent
sheaves on $\St^\eta$ form a module category over equivariant
coherent sheaves on the Steinberg variety $\St$.

One of the primary aims of this paper is to explain the close
relationship between $\St^\eta$ and the geometric parameter spaces
appearing in Vogan duality and Soergel's conjecture. In particular,
we will see that $\D$-modules on the geometric parameter spaces can
be recovered from equivariant coherent sheaves on $\St^\eta$.
Furthermore, the form of this relationship can be transported back
to the original (automorphic) module categories for which
equivariant coherent sheaves on $\St^\eta$ should provide a spectral
description. Namely, there is a precise form in which the loop
spaces of the geometric parameter spaces and their $S^1$-equivariant
geometry can be interpreted in terms of equivariant coherent sheaves
on $\St^\eta$ and their intrinsic categorical structures.


\subsection{Langlands parameters as loop spaces}\label{results}
The central theme of this paper is that the fundamental relationship
between the equivariant geometry of the flag variety $\cB$ and that
of the Springer variety $\wt G$ and Steinberg variety $\St$ is given
by the formalism of loop spaces.

To begin, observe that the quotient of the flag variety $\cB$ by the
group $G$ is nothing more than the classifying stack $pt/ B$ of a
Borel subgroup $B\subset G$. Thus the loop space $\cL(pt/B)$ is
immediately seen to be the adjoint quotient $B/B$. But this is
precisely the equivariant Springer variety
$$
\wt G / G \simeq \cL (G\backslash \cB).
$$

We observe that this simple statement generalizes to the Steinberg
variety.

\begin{thm} There is a canonical isomorphism of derived stacks
$$
\St/G\simeq \cL(B\backslash \cB)\simeq\cL(G\backslash (\cB\times
\cB)).
$$
\end{thm}

Via this statement, we can transport all of the structures of loop
spaces to the equivariant Steinberg variety. For example, it follows
that $\St/G$ carries a circle action, and the derived category of
quasicoherent sheaves on $\St/G$ carries a braided monoidal structure and
other string topology operations.

\medskip

Our primary application of the above theorem follows from restricting to
small loops and applying $S^1$-equivariant localization. In this way,
we recover the category of $B$-equivariant $\D$-modules on the flag
variety. More generally, by replacing small loops by alternative
formal completions, we obtain categories of Borel equivariant
$\D$-modules on flag varieties for various subgroups.

For any $(\alpha,\beta)\in H\times H$, we write $\St_{\alpha,\beta}$
for the inverse image in $\St$ of the formal neighborhood of
$({\alpha},{\beta})$ under the canonical projection. Note that
$\St_{\alpha,\beta}$ is nonempty if and only if $\alpha$ and $\beta$
are related by the Weyl group. In particular, we have the formal
Steinberg variety $\St_{e,e}$ corresponding to the identity
$(e,e)\in H\times H$.

Finally, for any $\alpha\in H$, we write $G_\alpha\subset G$ for the
reductive subgroup that centralizes $\alpha$, and $\cB_\alpha =
G_\alpha/B_\alpha$ for its flag variety.

\begin{thm}\label{complex crystalization}
There is a canonical quasi-equivalence of periodic dg derived
categories
$$
\Coh(\St_{e,e}/G)^{S^1}_{per}\simeq \D_B(\cB)\ot_\C\C[u,u\inv].
$$

More generally, for any $\alpha\in H$, and Weyl group element $w$,
there are canonical quasi-equivalences of periodic dg derived
categories
$$
\Coh(\St_{\alpha, w\alpha}/G)^{S^1}_{per} \simeq
\D_{B_\alpha}(\cB_\alpha)\ot_\C\C[u,u\inv].
$$
\end{thm}

It is worth pointing out that the theorem is not a direct
consequence of our previous results on the relation between
$S^1$-equivariant sheaves on Hochschild spaces and
$\D$-modules. For example, the adjoint quotient $\St_{e, e}/G$ is
not the Hochschild space of $B\backslash \cB$, but rather also
contains loops which are large in the unipotent direction. What we
have is an $S^1$-equivariant embedding
$$
\cH(B\backslash \cB) \hookrightarrow \St_{e,e}/G.
$$
One can show that restricting coherent sheaves along this embedding
gives an equivalence and then the theorem follows from our previous
results. A similar argument holds for a general parameter $\alpha$.

\medskip

An interesting aspect of the theorem is the general ``discontinuity"
of the objects appearing on the right hand side. From a geometric
perspective, the quotients $B_\alpha\backslash \cB_\alpha$ do not
form a nice family as we vary the parameter $\alpha$.
But the theorem says that the loop spaces of these quotients do fit
into the nice family formed by the equivariant Steinberg variety
$\St/G$. Here we should emphasize that we are thinking about $\St/G$
as a loop space, rather than thinking about it along the more
traditional lines of a cotangent bundle. In the discussion to
follow, we describe similar results for $\D$-modules on the
geometric parameter spaces for real reductive groups. In that
context, it is only the loop spaces which fit together into a nice
family, not the cotangent bundles.

\medskip

Now fix an algebraic involution $\eta$ of the group $G$. Our aim is
to describe a generalization of the above results for the Langlands
parameter space $\St^\eta$ and the geometric parameter spaces
$K_{\alpha,\sigma}\backslash \cB_\alpha$ appearing in Vogan duality
and Soergel's conjecture. The basic observation that underlies all
applications is that the equivariant Langlands parameter space
$\St^\eta/G$ can be naturally identified as a {path space}.

\begin{thm} The equivariant Langlands parameter space $\St^\eta/G$ is the derived space of paths
$$
\gamma= (\gamma_1,\gamma_2):[0,1] \to G\backslash (\cB \times \cB)
$$
satisfying the boundary equation
$$
(\gamma_1(0),\gamma_2(0)) = (\eta(\gamma_2(1)),\eta(\gamma_1(1)))
$$
\end{thm}

An alternative way to state the theorem is to say that $\St^\eta/G$
is the second component of the loop space of the quotient of
$G\backslash (\cB \times \cB)$ by the $\Z/2\Z$-action
$$
(B_1,B_2) \mapsto (\eta(B_2),\eta(B_1)).
$$
(To be precise, to recover $\St^\eta/G$, one must also equip such
loops with a trivialization of their associated $\Z/2\Z$-torsor.)
From this perspective, we see that there is a canonical $S^1$-action
of loop rotation on $\St^\eta/G$.

For any $\alpha\in H$, we write $\St^\eta_\alpha$ for the inverse
image in $\St^\eta$ of the formal neighborhood of  $\alpha$ under
the canonical projection. We can now state the loop interpretation
of the categories appearing as spectral
parameters in Soergel's conjecture. %

\begin{thm}\label{real cyclic crystalization}
For any $\alpha\in H$ there is a canonical quasi-equivalence of
periodic dg derived categories
$$
\Coh(\St^\eta_{\alpha}/G)^{S^1}_{per} \simeq \bigoplus_{\sigma\in
\Sigma({\eta,\alpha})} \D_{K_{\alpha,\sigma}}(\cB_\alpha)\ot_\C
\C[u,u^{-1}]
$$
where the right hand side is the periodic version of Soergel's
category of Langlands parameters.
\end{thm}

In parallel with the complex case, the adjoint quotient
$\St^\eta_\alpha/G$ is not the Hochschild space of the union of the
geometric parameter spaces $K_{\alpha,\sigma}\backslash\cB_\alpha$
appearing in the right hand side of the theorem. Rather the
Hochschild space of the union canonically sits inside of
$\St^\eta_\alpha/G$, and the restriction of coherent sheaves along
this embedding gives an equivalence.

\medskip

The above theorem gives a description of $\D$-modules on the
geometric parameter spaces $K_{\alpha,\sigma}\backslash\cB_\alpha$
as part of a nice family with respect to the paramater $\alpha$.
Namely, these categories can be recovered from the loop spaces of
$K_{\alpha,\sigma}\backslash\cB_\alpha$, and the loop spaces in turn
fit into the nice family formed by the equivariant Langlands
parameter space $\St^\eta_\alpha/G$. In this setting, it is crucial
that we sought such a uniform picture in the realm of loop spaces
rather than cotangent bundles.


\section{Applications}\label{applications}
In this section we outline how the results of this paper fit into
our ongoing project \cite{geometric Vogan,complex,base change} to
apply ideas from the geometric Langlands program to representation
theory of real groups, specifically to Vogan duality and Soergel's
conjecture, which give refined forms of the local Langlands program
over the reals.

In broad outline, we relate the local geometric Langlands program to
the real local Langlands program using two principles,
$S^1$-equivariantization and geometric base change. The local
geometric Langlands program describes module categories over loop
groups and their Hecke algebras in terms of coherent sheaves on
spaces of Langlands parameters. Equivariant localization for loop
rotation relates the loop group (and affine Hecke algebras) to the
group $G$ (and finite Hecke algebras), and coherent sheaves on
Langlands parameters to $\D$-modules on flag varieties of the dual
group. This latter step is the role of the current paper. Thus
representation theory of $G$ (the complex local Langlands program)
arises as the $S^1$-equivariantization (or ``string states") of
representation theory of $LG$. On the other hand, the geometric base
change conjecture relates the geometric Langlands programs over
complex and real curves. The result is the real local Langlands
classification identifying categories of representations of real
forms of $G$ through their local Langlands parameters, which are
$\D$-modules on dual flag varieties \cite{ABV,Soergel}.

In Section \ref{complex Langlands} we explain the application of the
localization principle in the complex case and the resulting duality
for finite Hecke categories \cite{complex}. In Section \ref{real
Langlands} we describe a real form of the geometric Langlands
conjecture on $\pline$ and its application to Vogan duality
\cite{geometric Vogan}. Finally, in Section \ref{base change} we
introduce the geometric base change conjecture, and explains how it
implies a strong form of Soergel's conjecture \cite{base change}.


\subsection{Ramified geometric Langlands on $\pline$}\label{complex
Langlands}
Fix a complex reductive group $G$ with Langlands dual
group $G^\vee$. When referring to objects associated to $G^\vee$, we
will often adjoin the superscript ${}^\vee$ to our usual notation
without further comment. So for example, as usual $\cB$ will denote
the flag variety of $G$, and $\cB^\vee$ the flag variety of
$G^\vee$.

\medskip

The equivariant Steinberg variety $\St^\vee/G^\vee$ admits a natural
interpretation as the space of $G^\vee$-local systems on $\pline$
with parabolic structure (simple poles and compatible flags) at the
points $0,\infty\in\pline$. In particular, the map to the adjoint
quotient $G^\vee/G^\vee$ gives the monodromy of a local system
around the equator. Under this interpretation, the $S^1$-action on
$\St^\vee/G^\vee$ by loop rotation coincides with the $\Cx$-action
induced by the standard rotation of $\pline$ fixing $0,\infty$. Note
that the $\Cx$-action reduces to an $S^1$-action since it is
infinitesimally trivialized by the local system structure.

\medskip

Next consider the moduli stack $\ramBun$ of $G$-bundles on $\pline$
equipped with flags at the points $0,\infty\in\pline$. The geometric
Langlands program predicts an intimate connection between the
derived category of coherent sheaves $\Coh(\St^\vee/G^\vee)$ and the
derived category of monodromic $\D$-modules $\D(\ramBun)_{mon}$. By
Bezrukavnikov's work, both of the above categories are module
categories for the affine Hecke category. On the one hand,
$\Coh(\St^\vee/G^\vee)$ is nothing more than the regular module
category. On the other hand, in~\cite{complex}, we show that
$\D(\ramBun)_{mon}$ is the dual module category (in a precise sense
which would take some space to spell out).

One can interpret the duality of affine Hecke module categories
$$
\D(\ramBun)_{mon} \leftrightarrow \Coh(\St^\vee/G^\vee)
$$
as the {\em tamely ramified geometric Langlands correspondence} on
$\pline$. A key property of the duality is that it respects
automorphisms of the curve $\pline$ fixing the points
$0,\infty\in\pline$. Namely, the $S^1$-action on
$\Coh(\St^\vee/G^\vee)$ by loop rotation is transported to the
$\Cx$-action on $\D(\ramBun)_{mon}$ induced by the standard rotation
of $\pline$. Note that here as well the $\Cx$-action also reduces to
an $S^1$-action since it is infinitesimally trivialized by the
$\D$-module structure.

\medskip

Consider for a moment the moduli stack $\ramBun_{mon}$ of
$G$-bundles on $\pline$ with unipotent level structure at the points
$0,\infty\in\pline$. One can interpret objects of
$\D(\ramBun)_{mon}$ as $\D$-modules on $\ramBun_{mon}$ that are
constructible along the fibers of the projection $\ramBun_{mon}\to
\ramBun$. A key observation is that the $\Cx$-action on
$\ramBun_{mon}$ by the standard rotation of $\pline$ does {\em not}
reduce to an $S^1$-action since its orbits contain nontrivial moduli
of objects. Rather the action reveals important structure as
summarized in the following statement.

\begin{observation}\label{automorphic fixed points}
The fixed points of the natural $\Cx$-action on $\ramBun_{mon}$ are
precisely the open locus $\ramBun_{mon}^{\circ}$ of trivial
$G$-bundles with parabolic structure.
\end{observation}

%

As explained in \cite{complex}, the general principle of
$S^1$-equivariant localization applies directly to the derived
category of $\D$-modules on $\ramBun_{mon}$. The localization of the
category of $\Cx$-equivariant objects is equivalent to the periodic
version of the derived category of $\D$-modules on the fixed points
$\ramBun_{mon}^\circ$. By Observation~\ref{automorphic fixed
points}, the fixed points can be identified with the quotient of a
product of monodromic flag varieties
$$
\ramBun^\circ_{mon} \simeq G\backslash (\cB_{mon}\times\cB_{mon})
$$
Here by the monodromic flag variety $\cB_{mon}$, we mean the moduli
of a Borel subgroup $B\subset G$, together with an element $h\in
B/U$.

\medskip

Combining  the above discussion with the results of this paper, we
can summarize the situation in the following schematic diagram.

$$ \CD
\text{\underline{Automorphic side} } @.
\text{\underline{Spectral side} }\\
  \D_{G} (\cB \times \cB)_{mon}
 @>{\text{complex local Langlands equivalence}}>>
\D_{G^\vee} (\cB^\vee \times \cB^\vee)
\\
  @V{\text{Observation~\ref{automorphic fixed points}}}VV
  @VV{\text{Theorem~\ref{complex crystalization}}}V \\
  \D(\ramBun^\circ)_{mon}  @. 
  \Coh(\St^\vee/G^\vee)^{S^1} \\
   @A{\text{$\Cx$--fixed points}}AA @AA\text{$S^1$--fixed points}A \\
   \D(\ramBun)_{mon}
   @>{\ \ \text{Ramified geometric Langlands}\ \ }>>
   \Coh(\St^\vee/G^\vee)
\endCD $$
\\

The equivalence on the bottom row, as discussed above, is a form of
the tamely ramified local geometric Langlands theorem of
Bezrukavnikov, i.e. Langlands duality for affine Hecke categories.
The equivalence of the top row is a Langlands duality for finite
Hecke categories, and a form of the complex local Langlands
classification. The existence of such an equivalence is a theorem of
Soergel~\cite{Soergel} whose proof is based on the Koszul duality
theorem of Beilinson, Ginzburg, and Soergel~\cite{BGS}. (Concretely,
it derives from a calculation of Ext groups of generators on both
sides.) Via Beilinson-Bernstein localization, the left hand side is
the category of Harish-Chandra bimodules with trivial generalized
infinitesimal character, and the equivalence is the complex case of
Soergel's general conjecture on Langlands parametrization of
categories of Harish-Chandra modules for real groups. The arguments
we have sketched here (and develop in detail in~\cite{complex})
provide a {\em canonical} construction and characterization of this
equivalence as well as a conceptual explanation for its existence.
These results can also be viewed as a tamely ramified and conceptual
version of the $S^1$-equivariant geometric Satake correspondence of
Bezrukavnikov and Finkelberg \cite{BezFink}, which relates the
$\Gm$-equivariant version of the derived Satake category to
Harish-Chandra bimodules for the dual group.


\subsection{Geometric Langlands for real groups}\label{real
Langlands}

Next we introduce real forms of $G$ into the geometric Langlands
correspondence. Fix once and for all a quasi-split conjugation
$\theta$ of $G$, and let $\eta$ be the combinatorially corresponding
algebraic involution of $\Gv$.

\medskip

Consider the antipodal conjugation of $\pline$, and note that it
exchanges the points $0,\infty\in\pline$. Thus together with the
conjugation $\theta$, it provides a real form $\ramBun_\theta$ of
the moduli stack $\ramBun$. Similarly, we have the corresponding
real form $\ramBun_{\theta, mon}$ of the monodromic moduli stack
$\ramBun_{mon}$. Observe that since $0,\infty\in\pline$ are
exchanged by the anitpodal conjugation, the natural projection
$\ramBun_{\theta, mon}\to \ramBun_{\theta}$ is a torsor for a single
copy of the universal Cartan $H$.

The restriction of the standard $\Cx$-action on $\pline$ to the
unitary circle $U(1)\subset \Cx$ preserves the antipodal
conjugation. Thus we have the induced $U(1)$-action on the above
moduli stacks. The orbits of the action on $\ramBun_\theta$ are
discrete, but those of the action on $\ramBun_\theta$ have moduli.

\begin{observation}\label{real automorphic fixed points}
The fixed points of the natural $U(1)$-action on $\ramBun_{\theta,
mon}$ are precisely the open locus $\ramBun_{\theta, mon}^{\circ}$
where the underlying $G$-bundle is trivial.
\end{observation}

Fix $\alpha\in H^\vee$, and consider the derived category
$\Sh(\ramBun_\theta)_\alpha$ of monodromic constructible sheaves on
$\ramBun_\theta$ with monodromicity $\alpha$. As explained
in~\cite{geometric Vogan}, the general principle of
$S^1$-equivariant localization applies in this setting: the
localization of the category of $U(1)$-equivariant objects of
$\Sh(\ramBun_\theta)_\alpha$ is equivalent to the differential
$\Z/2\Z$-graded version of the derived category
$\Sh(\ramBun^\circ_\theta)_\alpha$ of monodromic constructible
sheaves on the fixed points. Using Observation~\ref{automorphic
fixed points}, one can show that the fixed points are a union of
real quotients of monodromic flag varieties
$$
\ramBun^\circ_{\theta,mon} \simeq \coprod_{\theta'\in \Theta(\eta)}
G_{\R,\theta'}\backslash \cB_{mon}.
$$
Here the index set $\Theta(\eta)$ is precisely the one arising in
Vogan duality and Soergel's conjecture. Now there is an analytic
version of Beilinson-Bernstein localization due to
Kashiwara-Schmid~\cite{KS} that localizes representations of a real
form $G_{\R,\theta'}$ on the corresponding real quotient of the
monodromic flag variety $G_{\R,\theta'}\backslash \cB_{mon}$. Thus
for $\lambda\in\h^\vee$ with $\alpha=\exp(\lambda)$, we have an
identification of derived categories
$$
\Sh(\ramBun^\circ_\theta)_\alpha \simeq \bigoplus_{\theta'\in
\Theta(\eta)} \D({\HC}_{\theta',[\lambda]}).
$$
Here the right hand side (or rather its pro-completion) is the
derived category of Harish-Chandra modules appearing in Soergel's
conjecture.

\medskip

Now we can summarize our program to understand Soergel's conjecture
in the following schematic diagram. The left hand automorphic column
has been sketched in the preceding discussion. The right hand
spectral column follows from the results of this paper as described
in the overview. Finally, the bottom horizontal arrow is a
conjectural real geometric Langlands correspondence relating module
categories for the affine Hecke category. (Some intrinsic motivation
for the form of this statement will be given in the subsequent
section.)

$$ \CD
\text{\underline{Automorphic side} } @.
\text{\underline{Spectral side} }\\
\bigoplus_{\theta'\in \Theta(\eta)} \D({\HC}_{\theta',[\lambda]})
 @>{\text{Soergel conjecture}}>>
\bigoplus_{\sigma\in \Sigma({\eta,\alpha})}
\D_{K^\vee_{\alpha,\sigma}}(\cB^\vee_\alpha)
\\
  @V{\text{Localization of~\cite{KS}}}VV
  @VV{\text{Theorem~\ref{real cyclic crystalization}}}V \\
  \Sh(\Bunthetao)_{\alpha}  @. 
 \Coh(\St^{\vee,\eta}_{\alpha}/G^\vee)^{S^1}
 \\
   @A{\text{$U(1)$--fixed points}}AA @AA\text{$S^1$--fixed points}A \\
  \Sh(\Buntheta)_{\alpha}@>{\text{Real geometric Langlands conjecture}}>>
  \Coh(\St^{\vee,\eta}_\alpha/G^\vee) \\
\endCD $$
\\

By our other results, a construction of the real geometric Langlands
correspondence would resolve Soergels's conjecture. In the next
section, we will discuss how such a correspondence follows from a
conjectural {\em geometric base change} principle. This is something
we can currently verify holds on the level of Grothendieck groups.
Coupling it with the Kazhdan-Lusztig theorem on the affine Hecke
algebra, we obtain the following affine version of Vogan duality (or
real Kazhdan-Lusztig theorem).

%

\begin{thm}[\cite{geometric Vogan}]\label{affine Vogan}
(Affine Vogan Duality) For any $\alpha\in H^\vee$, there is a
canonical duality of modules for the affine Weyl group
$$
K(\Sh(\Buntheta)_{\alpha}))\leftrightarrow K(
\Coh(\St^{\vee,\eta}_\alpha/G^\vee)).
$$
\end{thm}

It is worth emphasizing that the proof of this theorem does not
depend on difficult categorical statements such as found in the work
of Bezrukavnikov, but only the easier original $K$-theoretic
Kazhdan-Lusztig theorem.

\medskip

As a corollary, we obtain a canonical and conceptual new proof of
Vogan duality.

\begin{corollary} (Vogan Duality)
Vogan's character duality isomorphism
$$
\bigoplus_{\theta'\in \Theta(\eta)} K({\HC}_{\theta',[\lambda]})
\quad \leftrightarrow \quad \bigoplus_{\sigma\in
\Sigma({\eta,\alpha})} K(\D_{K_{\alpha,\sigma}}(\cB_\alpha))
$$
follows by applying $S^1$-equivariant localization to Theorem
\ref{affine Vogan}.
\end{corollary}


\subsection{Geometric base change}\label{base change}
In this section, we sketch results from~\cite{base change}
explaining how the principle of base change from the Langlands
program manifests in the geometric setting. In particular, a
geometric form of base change provides a proof of
Theorem~\ref{affine Vogan}, and in particular, a conceptual proof of
Vogan duality. If geometric base change can be verified on $\pline$,
it will also provide a proof of Soergel's conjecture. More
generally, when it can be verified, it reduces the real geometric
Langlands correspondence to the usual complex version.

The following schematic diagram summarizes how base change fits into
our preceding discussion. The horizontal arrows have been discussed
in the two preceding sections. Our aim here is to explain the
vertical arrows. The right hand spectral base change is a result
from~\cite{base change}. The left hand automorphic base change is a
conjecture on the categorical level, and a theorem on the level of
Grothendieck groups~\cite{geometric Vogan}. To simplify what is a
very general discussion, we will suppress any further mention of
monodromic parameters.

$$ \CD
\text{\underline{Automorphic side} } @.
\text{\underline{Spectral side} }\\
  \Sh(\Buntheta)_{}@>{\text{Real geometric Langlands conjecture}}>>
  \Coh(\St^{\vee,\eta}/G^\vee) \\
   @V{\text{Conjectural base change}}VV @VV{\text{Spectral base change}}V \\
  \Sh(\ramBun)
   @>{\ \ \text{Ramified geometric Langlands}\ \ }>>
   \Coh(\St^\vee/\Gv)
\endCD
$$

\medskip

Recall that the equivariant Steinberg variety $\St^\vee/G^\vee$
admits a natural interpretation as the space of $G^\vee$-local
systems on $\pline$ with parabolic structure (simple poles and
compatible flags) at the points $0,\infty\in\pline$. Likewise, the
equivariant Langlands parameter space $\St^{\vee,\eta}/G^\vee$
admits a natural interpretation as the space of $G^\vee$-local
systems on $\pline$ with parabolic structure at the points
$0,\infty\in\pline$, and an $\eta$-twisted $\Z/2\Z$-equivariance
under the antipodal involution of $\pline$. It is often illuminating
to think of such an object as an $\eta$-twisted local system on the
quotient space $\R\pline = \pline/\Z/2\Z$ with a parabolic structure
at the point $0=\infty\in \R\pline$. From this perspective, it is
clear from Tannakian principles why we conjecture that
$\Coh(\St^{\vee,\eta}/G^\vee)$ should provide the spectral
description for $\Sh(\Buntheta)_{}$ under a real geometric Langland
correspondence. In what follows, we will also give a motivation for
this answer using the principle of base change.

\medskip

Consider the general situation of a covering map of curves $\hat
C\to C$ with Galois group $\Gamma$ so that $C= \hat C/\Gamma$.
Consider the stack of $G^\vee$-connections $\Conn_{G^\vee}(\hat C)$,
and its derived category of coherent sheaves
$\Coh(\Conn_{G^\vee}(\hat C))$. As explained in \cite{base change},
a simple spectral base change argument recovers
$\Coh(\Conn_{\Gv}(C))$ from natural operations on
$\Coh(\Conn_{\Gv}(\hat C))$. Namely, for each point of $\hat C$, we
have a tautological action of the tensor category $\Rep(\Gv)$ of
representations on $\Coh(\Conn_{\Gv}(\hat C))$, and imposing that
the action is identified for $\Gamma$-related points is precisely
the correct descent data.
One can go further and given an action of $\Gamma$ on $\Gv$ (or more
generally, a compatible action on a group-scheme over $\hat C$)
extend this picture to obtain a Galois-twisted version. It is also
worth remarking that the construction is compatible with respect to
automorphisms of the curves.

As a special case, taking $\hat C = \pline$, and $C=\R\pline$, with
$\Gamma=\Z/2\Z$ acting on $\pline$ via the antipodal map, and on
$\Gv$ via the involution $\eta$, we see that
$\Coh(\St^{\vee,\eta}/G^\vee)$ can be recovered from
$\Coh(\St^{\vee}/G^\vee)$ by spectral base change. Thus combining
this with the results of the current paper on $S^1$-equivariant
localization, we see that the derived category of equivariant
$\D$-modules on the geometric parameter spaces that appears in
Soergel's conjecture can be obtained from $\Coh(\St^\vee/\Gv)$ by
entirely formal categorical considerations.

\medskip

In \cite{base change}, we formulate and study the geometric version
of the base change principle in the automorphic setting. In general,
given a covering map of curves $\hat C\to C$, we arrive at a
conjecture that relates categories of $\D$-modules on the moduli
space $\Bun_{G}(\hat C)$ of $G$-bundles to $\D$-bundles on the
moduli space $\Bun_{G}(C)$. As in the spectral setting, the
construction is purely categorical involving only the natural Hecke
operators of the theory.

In the special case when $\hat C = \pline$, and $C=\R\pline$, with
$\Gamma=\Z/2\Z$ acting on $\pline$ via the antipodal map, and on $G$
via the conjugation $\theta$, we arrive at a conjectural way to
recover $\Sh(\ramBun_\theta)$ directly from $\Sh(\ramBun)$. It is
worth emphasizing that this statement together with the results of
the current paper says that we should be able to see all of the
complicated categorical structures in the representation theory of
real groups directly from the complex case by abstract nonsense. In
fact, in this special case, we are able to verify automorphic base
change on the level of Grothendieck groups, hence we already know a
large part of the combinatorics satisfies this principle. The
argument is the primary ingredient in the proof of
Theorem~\ref{affine Vogan}, and packages all of the combinatorics in
Vogan duality into a concise conceptual framework.

A proof of automorphic base change at the categorical level,
combined with the work of Bezrukavnikov in the complex case and the
results of the current paper, would provide a proof of Soergel's
conjecture. In fact, this line of argument gives an improved
formulation since all of the steps are canonical, and compatible
with Hecke actions. By contrast, Soergel conjectures an equivalence
via the existence of generators on both sides with isomorphic
endomorphism algebras.

\begin{thm}[\cite{base change}] A canonical,
Hecke-equivariant form of Soergel's conjecture follows from the
geometric base change conjecture for the antipodal map on $\pline$.
\end{thm}


\section{Loop and Hochschild Spaces}\label{loop section}

\nc{\DGA}{{\mathcal{DGA}}}

In this paper, all rings are assumed to be commutative, unital and
over a field $k$ {\em of characteristic $0$}.

\medskip

The theory of derived stacks provides a useful language to discuss
the objects  of representation theory that interest us. For the
reader's convenience, we have provided a brief Appendix
summarizing some of the basic motivation and terminology from the
theory of quasicategories and derived stacks. For derived stacks, we
refer the reader to To\"en's extremely useful survey~\cite{Toen},
and to the papers of To\"en-Vezzosi \cite{HAG1,HAG2} and Lurie
\cite{Lurie,DAG1,DAG2,DAG3}. In particular, we were introduced to
derived loop spaces by \cite{Toen}. We will not need any deep
statements from this theory, only the formalism that allows us to
perform basic constructions and to obtain well-behaved categories of
sheaves. We will work in the context of quasicategories
\cite{Joyal,Bergner} (or $\infty$-categories in the language of
\cite{topoi}). As explained in the Appendix (see the survey
\cite{Bergner}), this is one of many equivalent categorical contexts
for homotopical algebra and geometry which lie in between the coarse
world of homotopy categories and the fine world of model categories.
In particular, the Dwyer-Kan simplicial localization of a model
category provides a primary source of quasicategories.

In this section, we collect the definitions and basic properties of
the loop space $\cL X$ of a derived stack. We then focus on
the space of small loops or Hochschild space $\cH X$ obtained by
formally completing $\cL X$ along the constant loops $X$. One
motivation for introducing small loops is that they are local in $X$
in a suitable sense. With our intended applications in mind, we content
ourselves with
the concrete situation when $X$ is a smooth Artin stack
(though the discussion surely holds in far greater generality). We show
that $\cH X$ is isomorphic to the formal completion of the odd
tangent bundle along its zero section. 
Passing to functions
gives an isomorphism of the Hochschild homology and de Rham
algebra of $X$. One can view this as a version of the
Hochschild-Kostant-Rosenberg theorem in the context of stacks.
Although not needed in what follows, we also review
other well-known structures such as the Atiyah
bracket on Hochschild cohomology and its interpretation in this
context.


\subsection{The loop space $\cL X$}

Let $\cS$ denote the quasicategory of simplicial sets, or
equivalently (compactly generated Hausdorff) topological spaces. Let
$\SCA_k$ denote the quasicategory of simplicial commutative unital
$k$-algebras, or equivalently connective commutative differential
graded $k$-algebras.
An object of the quasicategory
of derived stacks over $k$ is a functor
$$
X:\SCA_k\to \cS
$$
which is a sheaf in the \'etale topology.

Some natural classes of derived stacks are provided by derived
schemes (in paricular, representable functors given by
affine derived schemes),
Artin stacks, and topological spaces. For the latter, any compactly
generated Hausdorff topological space $K$ defines a derived stack
given by the sheafification of the constant functor
$$
K:\SCA_k\to \cS 
\qquad K(R)= K.
$$
To connect with other combinatorial models, it is often convenient
to choose a simplicial presentation of $K$ (for example, that given
by singular chains) and consider $K$ as a functor to simplicial
sets. Of course, any two simplicial presentations lead to equivalent
stacks.

\medskip

Given derived stacks $K,X$, morphisms of sheaves form a derived
mapping stack $ \RHom(K,X)$.
When $K$ is a topological space,
we can think of $\RHom(K,X)$ as a collection of equations imposed on
copies of $X$. One can check that for $K$ a
finite
simplicial set and $X$ a derived 
Artin stack, the derived
mapping stack $\RHom(K,X)$ is also a derived 
Artin stack. (The reader could consult the Appendix for a discussion
on what it means for a derived stack to be Artin.)

\medskip

In this paper, we will focus on the locally constant stack given by the
circle $K=S^1$, which is identified with the classifying space
$B\Z$. In this case, we refer to the corresponding derived mapping
stack $\RHom(S^1, X) $ as the loop space of $X$ and denote it
by $\cL X$. Roughly speaking, we take a copy of $X$ and impose the
equation that every point of $X$ must be equal to itself.

To make this concrete, we can choose a
simplicial presentations of $S^1$.
For example, a particularly small
presentation of $S^1$ as a simplicial set has two $0$-simplices, two
nontrivial $1$-simplices, and no nontrivial higher simplices. This
leads to the usual model of the loop space as the derived
fiber product
$$
\cL X \simeq X \times^{\R}_{X\times X} X
$$
along the diagonal maps. Thus the pushforward to $X$ of functions on $\cL
X$ is given by the derived tensor product
$$
\cO_{X} \otimes^{\bbL}_{\cO_{X\times X}} \cO_{X}.
$$
When $X$ has affine diagonal, so that
$\cL X$ is affine over $X$, we can think of $\cL X$ as the
spectrum over $X$ of the derived tensor product.

Two examples are useful to keep in mind.
When $X$ is an ordinary affine scheme, the component ring
$\pi_0(\cO_{\cL X})$ is nothing more than $\cO_X$, thus the
underlying ordinary scheme of $\cL X$ is simply $X$ itself, and so $\cL
X$ is a purely derived enhancement. On the other hand, if
$X=BG$ is the classifying space of an algebraic group $G$, it is
easy to see that $\cL X=G/G$ is the adjoint quotient stack with trivial derived structure.

\medskip

Of the many interesting structures on $\cL X$, we will concentrate
on the $S^1$-action given by loop rotation
$$
S^1\times \cL X\to \cL X.
$$
Connes' theory of cyclic objects~\cite{Connes} provides a convenient
algebraization of $S^1$ and more generally of $S^1$-spaces (see \cite{Jones}
for an application to free loop spaces and \cite{Loday} for a
detailed exposition). Consider $S^1$ as the unit circle in the
complex plane, and let $\Z_{n+1}=\{z\in S^1| z^{n+1}=1\}$ denote the
$(n+1)$st roots of unity. Connes' cyclic category $\Cyc$ is the
category with objects finite ordered sets $\bn =\{0,1,\ldots, n\}$,
and morphisms homotopy classes of continuous, order preserving,
degree one maps of pairs
$$
\Cyc(\bn,\bm) = \left [s:  (S^1,\Z_{n+1}) \to (S^1,\Z_{m+1}) \right
].
$$
Here a map $s:S^1\to S^1$ is said to be order preserving if any lift
$\widetilde s:\R \to \R$ is non-decreasing. The geometric
realization of the simplicial set
$$
\Lambda=\Cyc(-,[0])
$$
is homeomorphic to $S^1$.

This presentation of $S^1$ leads to
the familiar model of $\cL X$ as a
cocyclic space with $n$-cosimplices given by
the products $X^{n+1}$ (with the usual diagonal and projection structure maps).
The pushforward
to $X\times X$ of functions on $\cL X$ is given by the usual cyclic
complex of Hochschild chains
$$
\cC^{-\bul}(\cO_X)= \cB^{-\bul}(\cO_X) \otimes_{\cO_{X\times X}}
\cO_X
$$
where the bar resolution of $\cO_X$ is given by
$$
\cB^{-\bul}(\cO_X) = \cdots \to \cB^{-2}(\cO_X) \stackrel{\del}{\to}
\cB^{-1}(\cO_X) \stackrel{\del}{\to} \cB^{0}(\cO_X)
$$
with terms
$$
\cB^{-q}(\cO_X) = \cO_{X^{(q+2)}} = \cO_X^{\otimes(q+2)} = \cO_X
\otimes\cdots\otimes \cO_X,
$$
with $\cO_{X \times X}$-module structure
$$
(a_\ell \otimes a_r) \cdot (r_0\otimes \cdots \otimes r_{q+1}) =
a_\ell r_0\otimes \cdots \otimes r_{q+1} a_r,
$$
and differential
$$
\del(r_0\otimes \cdots \otimes r_{q+1}) = \sum_{i=0}^q (-1)^i
r_0\otimes\cdots \otimes r_i r_{i+2} \otimes \cdots\otimes r_{q+1}.
$$

\subsection{The Hochschild space $\cH X$}\label{hochschild space}

Let $X$ denote a derived stack and let $\cL X$ denote its loop
space. Consider the canonical projection $\cL X\to X$ given by the
evaluation of loops at the base point. One might naively hope that
the loop space functor satisfies some form of descent with respect
to maps to $X$. But as in more traditional contexts, it is
impossible to realize all loops by gluing together local loops. We
can find a local version of loops
if we restrict from all loops to ``small loops".

To make this precise, consider the canonical map $X\to \cL X$ that
sends a point to the constant loop at that point.

\begin{defn} The {\em Hochschild space} (or small loop space) $\cH X$ of a derived 
stack $X$ is the derived stack $\cH X=\wh{\cL X}_X$ obtained as the
formal completion of the loop space $\cL X$ along the constant loops
$X\to \cL X$.
\end{defn}

Note that since the constant loops $X\to \cL X$ are preserved by loop
rotation, the $S^1$-action on $\cL X$ descends to an action on the
Hochschild space $\cH X$.

\medskip

To return to our two previous examples,
when $X$ is an ordinary affine scheme, the Hochschild space $\cH X$
coincides with the loop space $\cL X$.
On the other hand,
for $X=BG$ the classifying space of an algebraic group $G$,
the Hochschild
space $\cH BG = \wh G/G$
is the stack quotient of the formal group $\wh{G}$
by the adjoint action of $G$

\medskip

For the reader's convenience, let us spell out what it means to take the formal completion in the
language of functors of points. We will give an interpretation in
which we separate the derived aspect of the situation from the
formal. For $R\in \SCA_k$, let $S=\Spec R$ be the affine derived
scheme given by the representable functor
$$
\Hom(R, -):\SCA_k \to \cS.
$$
The components $\pi_0(R)$ form an ordinary discrete commutative
algebra and we have a canonical projection $R\to \pi_0(R)$. Thus the
ordinary affine scheme $S_0 =\Spec \pi_0(R)$ comes equipped with a
canonical map $S_0\to S$ of affine derived schemes. In this way, we
may think of $S$ as a kind of ``derived thickening" of $S_0$.

Consider the nilradical $\cN\subset \pi_0(R)$ and the corresponding
reduced affine scheme $S_{0,r} = \Spec \pi_0(R)/\cN$. Via the
canonical map $S_{0,r}\to S_0$, we may think of $S_0$ as a ``formal
thickening" of $S_{0,r}$. Now given a map of derived stacks $X\to
Y$, the formal completion $\wh{Y}_{X}$ of $Y$ along $X$ assigns to
the affine derived scheme $S$ the space of homotopy commutative
diagrams
$$
\begin{array}{ccc}
S & \to & Y \\
\uparrow & & \uparrow \\
S_{0,r} & \to & X
\end{array}
$$
To be precise, maps from the test object $S$ into the formal
completion $\wh Y_{X}$ are given by the homotopy fiber product
$$
\Hom(S, \wh Y_X) = \Hom(S,Y) \times^{\mathbb R}_{\Hom(S_{0,r},Y)}
\Hom(S_{0,r},X).
$$


\newcommand{\completedSym}{\wh{\operatorname{Sym}}^{\bullet}}

\subsection{Case of Artin stacks} In what follows,
we will restrict our study of the loop space $\cL X$ and Hochschild space $\cH X$ to the
situation where $X$ itself is a smooth Artin stack with affine
diagonal. Not only will this assumption simplify the discussion, but
our intended applications in representation theory fit into this
context. Our need to consider derived stacks arises from the fact
that they appear as a result of the loop space construction.

\medskip

Let $X$ be a smooth Artin stack with affine
diagonal.
Consider a presentation of $X$ with objects and morphisms given by
smooth schemes $X_0$ and $X_1$ respectively,
and groupoid structure maps denoted as follows
$$
\ell,r:X_1\to X_0
\qquad
e:X_0\to X_1
\qquad
m:X_1 \times_{X_0} X_1 \to X_1
\qquad
i:X_1\to X_1.
$$

Consider the pushforward to $X$ of the structure sheaf $\cO_{\cL X}$.
By construction, it can be realized as the descent from $X_0$ of the derived tensor product
$$
\cO_{X_1} \otimes^\bbL_{X_0\times X_0} \cO_{\Delta_{X_0}}
$$
where $X_1$ maps to $X_0\times X_0$ by the product $\ell\times r$,
and $\Delta_{X_0}$ denotes the diagonal of $X_0\times X_0$.
In other words, loops are thought of as pairs of a $1$-simplex and a $0$-simplex
such that
the two ends of the $1$-simplex are glued to the $0$-simplex.

\medskip

The first result we will need is a characterization when the loop space $\cL X$
has trivial derived structure.

\begin{prop}\label{dg structure of loops}
The loop space $\cL X$ has trivial derived structure if and only if
the isomorphism classes of objects $X$ are discrete.
\end{prop}

\begin{proof}
It is convenient to rewrite the pushforward to $X$ of the structure sheaf $\cO_{\cL X}$
as the descent of the derived tensor product
$$
\cO_{X_1} \otimes^\bbL_{X_0\times X_1} \cO_{X_1}
$$
where $X_1$ maps to $X_0\times X_1$ by the product maps $\ell\times \id_{X_1}$
and $r\times \id_{X_1}$.
This can be viewed as the structure sheaf of the derived intersection of the subschemes
$\Gamma_\ell,\Gamma_r\subset X_0\times X_1$ given by the graphs of $\ell, r$ respectively.
Here loops are thought of as pairs of $1$-simplices that are equal
and such that the left end of the first is glued to the right end of the second.

Now our assertion will follow from a simple dimension count.
Let $n_0$ and $n_1$ be the dimensions of $X_0$ and $X_1$ respectively.
On the one hand,
the expected dimension of the intersection $\Gamma_\ell\cap \Gamma_r$ inside of $X_0\times X_1$
is given by $n_1+n_1 - (n_0 + n_1) = n_1 - n_0$. On the other hand,
each isomorphism class of objects of $X$
contributes a subscheme of precisely dimension $n_1-n_0$ to the intersection.
Thus the intersection has the expected dimension if and only if there is no nontrivial moduli
of isomorphism classes of objects.
\end{proof}

The next result we will need is an identification of the Hochschild space $\cH X$
with the completed odd tangent bundle.

Let $\BT_{X_0}$ denote the tangent sheaf of $X_0$, and let
$\fg_{X_1}$ denote the Lie algebroid on $X_0$ associated to the
groupoid. Recall that the tangent complex of $X$ is the descent from $X_0$ of
the complex built out of the action map
$$
\alpha:\fg_{X_1}[1] \to \BT_{X_0}.
$$
By definition, the odd tangent bundle of $X$ is the descent from $X_0$ of the
spectrum of the symmetric algebra of the shifted cotangent complex
$$
\BT_X[-1] = \Spec(\Sym_{\cO_X}(\Omega_X^1[1]
\stackrel{\alpha^*}{\to} \fg^*_{X_1})).
$$
We write $\wh{\BT}_X[-1]$ for the completion of $\BT_X[-1]$ along
the base $X$ and call it the completed odd tangent bundle.

\begin{prop}\label{hochschild is completed odd tangent}
For $X$ a smooth Artin stack with affine diagonal, the Hochschild
space $\cH X$ is canonically isomorphic to the completed odd tangent
bundle $\wh{\BT}_X[-1]$.
\end{prop}

\begin{proof}
Let $\wh X_{1, X_0}$ denote the completion of $X_1$ along the
unit morphism $e:X_0\to X_1$.
Consider the pushforward to $X$ of the structure sheaf $\cO_{\cH X}$.
By construction, it can be realized as the descent from $X_0$ of the
completed tensor product
$$
\cO_{\cH X} = \cO_{\wh X_{1, X_0}}
\wh \otimes^\bbL_{\cO_{X_0 \times
X_0}} \cO_{\Delta_{X_0}}.
$$
Consider the formal completion of the diagonal $\Delta_{X_0}$ inside of $X_0\times X_0$, and
the associated completed Koszul resolution of the sheaf of functions
$\cO_{\Delta_{X_0}}$. Then performing the completed tensor product,
we obtain a complex given by the completed symmetric product
$$
\completedSym_{\cO_X}(\Omega_X^1[1] \stackrel{\alpha^*}{\to}
\fg^*_{X_1}).
$$
This is precisely the ring of functions on the completed odd tangent
bundle $\wh\BT_X[-1]$.
\end{proof}

\begin{corollary}\label{dg structure of hochschild}
The Hochschild space $\cH X$ has trivial derived structure if and only if
each irreducible component of $X$ contains a dense isomorphism class of objects.
\end{corollary}

\begin{proof}
The second postulate is equivalent to the injectivity of the dualized action map $\alpha^*$.
\end{proof}

For our purposes, the key consequence of the proposition is that the Hochschild space $\cH X$ is a
local object in the following sense. Consider the simplicial scheme
$X_\bul$ with $0$-simplices given by $X_0$, and for $k>0$,
$k$-simplices given by the fiber products
$$
X_k = X_1 \times_{X_0} \cdots \times_{X_0} X_1 \qquad \mbox{ with $k$
factors.}
$$
Functions on the completed odd tangent bundle 
can be calculated as the limit of functions on the
completed odd tangent bundles 
of the simplices.
Thus by the proposition, functions on the
Hochschild space $\cO_{\cH X}$ can be calculated as the limit of functions on the
Hochschild spaces $\cO_{\cH X_k}$ of the simplices.
(It is worth remarking, as pointed out by
J. Lurie, that this descent for functions is not valid for the
Hochschild space itself.) Observe that this presentation is
compatible with the $S^1$-action of loop rotation:
the $S^1$-action on $\cO_{\cH X}$ is the limit of the
$S^1$-actions on the terms $\cO_{\cH X_k}$.

By the above discussion, to understand the local geometry of $\cH
X$, it will usually suffice to study the case when $X$ is simply a
smooth affine scheme. In this case, the odd tangent bundle
 and the completed odd tangent bundle coincide (since
both have $X$ as their underlying schemes). Both are the spectrum of
the de Rham algebra of differential forms $\Omega^{-\bul}_X$.
Furthermore, the
identification of the proposition is nothing
more than the Hochschild-Kostant-Rosenberg theorem. Under the
functor from cyclic modules to mixed modules (see~\cite{Loday}), the
$S^1$-action on $\cO_{\cH X}$ goes over to the de Rham differential
$d$ on $\Omega^{-\bul}_X$.

\subsection{Lie structure of loop spaces}\label{structures}
In this informal section, we mention further structures on loop
spaces and place them in our current context.
The discussion will not be used in the remainder of the paper.

\medskip

The circle $S^1$ (equipped with a fixed basepoint $1\in S^1$) has a
natural comultiplication in the category of pointed spaces
$$
S^1 \to S^1 \vee S^1.
$$
Given any derived stack $X$, the loop space $\cL X$ inherits a
multiplication
$$
\cL X \times_X \cL X\to\cL X
$$
from the comultiplication on $S^1$. As usual, this multiplication is
not associative but rather fits into an $A_\infty$-monoid structure
over $X$.

The circle $S^1$ also has a natural time-reversal automorphism
fixing the base-point. Thus the loop space $\cL X$ inherits a
parametrization-reversal automorphism.

We like to summarize the situation by thinking of the loop space
$\cL X$ as a Lie group and the Hochschild space $\cH X$ as its
formal group.
Taking this perspective, 
it is natural to ask about its Lie algebra. One can interpret this
as the odd tangent bundle $\BT_X[-1]$ equipped with its canonical
Lie algebra structure
$$
\BT_X[-1] \otimes \BT_X[-1] \to \BT_X[-1]
$$
given by the Atiyah class~\cite{Kap RW}.
The analogue of the enveloping algebra, or
 space of distributions on $\cL X$ supported along $X$,
 is the usual Hochschild cochain complex
$$
\R\IntHom_{\cO_{X \times X}}(\cO_{\Delta_X},\cO_{\Delta_X})
$$
The Yoneda product of Exts gives an $A_\infty$-multiplication which
agrees with the convolution structure induced by the multiplication
of loops. This picture was explained by Markarian \cite{Mark} and
furthered by Ramadoss \cite{Rama1,Rama2} and Roberts-Willerton
\cite{RW}. From this perspective, the Hochschild-Kostant-Rosenberg
isomorphism becomes the Poincar\'e-Birkhoff-Witt isomorphism for the
Lie algebra $\BT_X[-1]$.

\section{Sheaves on loop spaces}\label{sheaves section}

In this section we consider the differential graded (dg) derived
category of quasicoherent complexes on the derived loop space $\cL
X$ of a smooth Artin stack and the categorical action of $S^1$
induced by loop rotation. We begin in Section \ref{dgcats} with a review of the construction
of dg derived categories of quasicoherent complexes on a derived
stack. In Section \ref{equivariant}, we
describe categories of $S^1$-equivariant sheaves on stacks in a
fashion inspired by the Koszul duality picture of Goresky, Kottwitz
and MacPherson \cite{GKM}. After a review of the Koszul dual descriptions
of de Rham modules and
$\D$-modules on stacks in Section \ref{cyclic D-modules},
we show in Section
\ref{D-modules} that periodic $S^1$-equivariant sheaves on
the Hochschild space $\cH X$ of a smooth Artin stack $X$ are
identified with periodic
$\D$-modules on $X$.

\subsection{Quasicoherent sheaves on a derived stack}\label{dgcats}
In this section, we briefly summarize some of the key definitions
and properties concerning dg categories and their construction from
derived stacks. We refer the reader to \cite{Keller,Toen} for
excellent overviews. We are indebted to Bertrand To\"en for very
helpful explanations. We refer to \cite{DAG2,DAG3} for the theory of
monoidal and symmetric monoidal quasicategories, specifically the
notions of algebra objects, module categories over algebra objects,
and tensor product of module objects in the commutative case (see
also \cite{Schwede Shipley} and references therein for the more
familiar theory in the context of model categories).

\medskip

A dg category over $k$ is a category enriched over dg $k$-vector
spaces. We remind the reader that throughout this paper $k$ is assumed to be a field of
characteristic zero; this significantly simplifies the associated
homotopy theory. Our basic examples
of dg categories are obtained by localizing quasi-isomorphisms in the category of
complexes in an abelian category $\cA$ as explained by Keller
\cite{Keller} and Drinfeld \cite{Drinfeld}. We will consistently
abuse standard terminology by referring to the result of this
construction 
as the dg derived category of the underlying abelian
category $\cA$.
It is worth pointing out from the start that not all of our dg derived categories
will arise in this manner.

There is a notion of quasi-equivalence of dg categories, mimicking
the notion of quasi-isomorphism of complexes: a quasiequivalence
induces equivalences of homotopy categories. We would like to work
with dg categories up to quasiequivalence. More formally, dg
categories admit a model category structure \cite{Tabauda} in which
quasi-equivalences are the weak equivalences, giving rise to a
quasicategory (the Dwyer-Kan simplicial localization) in which quasi-equivalences
have been inverted. With his model structure in mind,
we will construct many of our dg categories as
limits of diagrams of dg categories.

\medskip

To a derived stack $Z$, there is assigned a
dg category $\QCoh(Z)$ which we call the dg derived category of quasicoherent
complexes on $Z$ (see \cite[p.36]{Toen}).
Let us briefly recall the construction of $\QCoh(Z)$.

First, consider a simplicial commutative $k$-algebra $A$, and the
representable affine derived scheme $Z=\Spec A$.
Via the normalization functor, we may think of $A$ as a connective
commutative differential graded algebra. With this understanding, we
take $\QCoh(Z)$ to be the dg derived category of dg modules over $A$.
In other words, we take the localization of dg modules over $A$ with respect
to quasi-isomorphisms.

In general, any derived stack $Z$ can be written as a colimit of a
diagram of affine derived schemes $Z_\bullet$. Then we take $\QCoh(Z)$
to be the limit of the corresponding diagram of dg categories $\QCoh(Z_\bullet)$.
We can think of objects of $\QCoh(Z)$ as collections of
quasicoherent complexes $\cF^\bullet$ on the terms $Z_\bullet$
together with compatible collections of quasi-isomorphisms between their pullbacks under the
diagram maps.


In our applications, we are interested in $\QCoh(Z)$ for derived Artin stacks $Z$
with affine diagonal. In this case we can calculate $\QCoh(Z)$ by
traditional simplicial descent of derived categories as explained
for example in \cite[Section 7.4]{BD Hitchin}. Choose an atlas
$$
p:Z_0 \to Z,
$$
such that $Z_0$ is an affine derived scheme, and consider the
resolution of $Z$ by the simplicial affine derived scheme $Z_\bul$ with
$k$-simplices given by the fiber products
$$
Z_k = Z_0 \times_Z^{\bbL} \cdots \times_Z^{\bbL} Z_0 \qquad \mbox{ with
$k$ terms}.
$$
Quasicoherent sheaves on the simplices form a cosimplicial dg
category, and we take $\QCoh(Z)$ to be its totalization. Concretely,
we can think of objects of $\QCoh(Z)$ as collections of
quasicoherent sheaves $\cF^k$ on the simplices $Z_k$ together with
compatible quasi-isomorphisms between their pullbacks under the simplicial structure
maps. In other words, sheaves on $Z$ are described by modules for
the cosimplicial commutative dg algebra of functions on the simplices.

Finally for a map $p:X\to Y$ we have the usual pullback functor
$p^*:\QCoh(Y)\to \QCoh(X)$ and its right adjoint, the pushforward
$p_*:\QCoh(X)\to \QCoh(Y)$.

\medskip

We will be most interested in the dg derived category $\Coh(Z)$ of quasi-coherent sheaves
with finitely generated cohomology. Since one can define quasi-coherent sheaves
with respect to smooth test maps, it makes sense to consider this property.




\subsection{Models for equivariant sheaves}\label{equivariant} In this section, we consider equivariant
sheaves on derived stacks with group actions. Our point of view is
inspired by Koszul duality (specifically by \cite{GKM,Allday
Puppe}). The basic idea is that to give a space $Z$ with the action
of a group $G$ is the same as to give a space $Z/G$ with a map to
$BG$. The space $Z/G$ is the total space of the $Z$-bundle over $BG$
associated to the universal $G$-bundle $EG$ and $Z$ is the fiber of
this bundle. After linearization, the constructions $Z \mapsto Z/G$
and $Z/G\mapsto Z$ become invariants over the linearization of $G$ (which
will be a group algebra) and coinvariants over the linearization of $BG$
(which will be an equivariant cochain complex) respectively.
In what follows, we will concentrate on the case $G=S^1$.

\medskip

Goresky, Kottwitz and
MacPherson \cite{GKM} modify the grading conventions of BGG Koszul
duality to obtain an equivalence between the homotopy theories of
(bounded below) dg modules over the symmetric algebra
$$
\bS=H^*(BS^1)=k[u] \quad \mbox{ with $|u|=2$}
$$
and (bounded below) dg modules over the exterior algebra
$$
\bLa=H_{-*}(S^1)=k\oplus k\cdot\lambda \quad \mbox{ with
$|\lambda|=-1$}
$$
preserving subcategories of complexes with finitely generated
cohomology. Moreover, it is shown in~\cite{GKM} (in a geometric
setting) that the $\bS$-module of equivariant global sections of an
$S^1$-equivariant sheaf $\cF$ on an $S^1$-space $X$ corresponds to
the $\bLa$-module of ordinary global sections of $\cF$.


\medskip

Now consider a derived stack $Z$ with an action of $S^1$. This consists
of an action map
$$
act:S^1\times Z\to Z
$$
with coherent associativity and unit axioms. Equivalently, we can
give the derived stack $Z/S^1$ with a map to the classifying stack
$BS^1$, and an isomorphism
$$Z/S^1\times_{BS^1} ES^1\risom Z$$
where $ES^1\to BS^1$ is the universal $S^1$-bundle. In cases of interest
(such as loop spaces of Artin stacks) this action can be concretely
modeled by giving a cyclic affine derived scheme. Intuitively, a
sheaf on $Z\times S^1$ is simply a sheaf on $Z$ with an
automorphism, given by the monodromy along $S^1$, so that an action
of $S^1$ on $Z$ manifests itself at the categorical level as an
automorphism of the identity functor of $\QCoh(Z)$.\footnote{More
generally, a coaction $act^*:\cC\to \cC\otimes\QCoh(S^1)$ on a dg
category $\cC$ is equivalent to the data of an automorphism of the
identity functor $ m\in \Aut(\on{Id}_{\cC}). $}

By definition, we take the dg derived category $\QCoh(Z)^{S^1}$ of
$S^1$-equivariant quasicoherent sheaves on $Z$ to be the dg derived
category $\QCoh(Z/S^1)$ of the quotient derived stack $Z/S^1$. To
obtain an explicit model, we present $ES^1$ by its standard
simplicial model and calculate $Z/S^1=Z\times_{S^1} ES^1$ by the Borel
construction. We obtain $\QCoh(Z/S^1)$ by taking the totalization of
the corresponding cosimplicial dg category of quasicoherent sheaves
on the simplices. Informally, the equivariant category is the
homotopy equalizer of the identity and the monodromy operator on
$\QCoh(Z)$.

The equivariant derived category can be calculated by descent for
the map $p:Z\to Z/S^1$. Namely, the canonical adjunction between
$p_*$ and $p^*$ gives rise to a comonad on $\QCoh(Z)$. Since $p^*$
is conservative and we have enough colimits, the hypotheses of the
Barr-Beck theorem are satisfied (see \cite{DAG2} for the monadic
formalism and Barr-Beck theorem in the context of quasicategories).
The result is the following:

\begin{prop}\label{comonad}
The category $\QCoh(Z/S^1)$ is equivalent to the category of
comodule objects in $\QCoh(Z)$ for the coalgebra $p_*act^*$ in
$\End(\QCoh(Z))$. Equivalently, it is realized as comodule objects
for the coalgebra $(p\times act)_*\Oo_{S^1\times Z}$ in
$\QCoh(Z\times Z)$ with its convolution monoidal structure.
\end{prop}

As a result, $S^1$-equivariant quasicoherent sheaves on $Z$ are the
same thing as sheaves with a coaction of the cochain coalgebra
$C^*(S^1)$ lifting its action on $\Oo_Z$.

\medskip

As an illustration, consider the above descriptions of the
equivariant category $\QCoh(Z/S^1)$ in the case $Z=\Spec k$ so that
$Z/S^1 = BS^1$. It is traditional in this setting to regard comodule
objects for the formal coalgebra $C^*(S^1)$ rather as module objects
for the homology of the circle $\bLa= H_{-*}(S^1)=k\oplus
k\cdot\lambda$ with $|\lambda|=-1$.
Using the Dold-Kan correspondence, we can also identify $\QCoh(BS^1)$
with the quasicategory of cyclic $k$-modules. As a result, we obtain
a variant of the result of Dwyer-Kan \cite{Dwyer Kan} giving an
equivalence of quasicategories between $\bLa$-modules and cyclic
$k$-modules.


More generally, let us spell out the statement of Proposition \ref{comonad}
in the case of an affine derived scheme $Z=\Spec A$ with an
$S^1$-action. This will be the only form of the assertion used in what follows.
First, we may consider $A$ as a simplicial associative
rather then commutative algebra since this does not affect the
category of $A$-modules. Next, we may express the $S^1$-action by
considering $A$ as a cyclic associative $k$-algebra since cyclic
objects in a quasicategory model objects with an $S^1$-action. The
$S^1$-action on $Z$ gives rise to a coaction of $C^*(S^1)$ on $A$,
or equivalently an action of $\bLa$. We may therefore consider the
algebra $A[\bLa]$ generated by $A$ and $\bLa$ (in other words,
obtained by adjoining the generator $\lambda\in \bLa^{-1}$ to $A$).

\begin{corollary}\label{cyclic to Lie}
Let $A$ be a cyclic commutative $k$-algebra, and let $A[\bLa]$ be
the dg algebra generated by the $S^1$-action on $A$. Then the
equivariant dg derived category $\QCoh(\Spec A)^{S^1}$ is
quasiequivalent to the dg derived category $A[\bLa]-\module$.
\end{corollary}


\subsection{Equivariant sheaves and de Rham modules}\label{cyclic
D-modules} Let $X$ be a smooth Artin stack with affine diagonal. Recall that its
Hochschild space $\cH X$ is a derived Artin stack and comes equipped
with a canonical $S^1$-action given by loop rotation. We are now
ready to give a concrete description of the dg derived category $
\QCoh(\cH X)^{S^1} $ of $S^1$-equivariant quasicoherent complexes on
$\cH X$.

\medskip

Choose a smooth simplicial presentation $X_\bul\to X$ such that each of the simplices $X_k$ is a
smooth affine scheme. Since we assume that $X$ has affine diagonal, we can choose
a smooth affine cover $X_0\to X$ with $X_0$ affine, and then take $X_\bul$ to be the Cech nerve
with $k$-simplices given by the fiber products
$$
X_k = X_0 \times_{X} \cdots \times_X X_0
\qquad
\mbox{ with $k+1$ factors}.
$$

Consider the limit dg derived category
$\QCoh(\cH X_\bul)$ of compatible quasicoherent sheaves on the Hochschild spaces $\cH X_k$
of each of the simplices $X_k$.
Consider as well the limit dg derived category
$\Coh(\cH X_\bul)$ of compatible coherent sheaves on the Hochschild spaces $\cH X_k$
of each of the simplices $X_k$.

\begin{prop}
The canonical $S^1$-equivariant map $\pi:\cH X_\bul {\to}
\cH X$ induces an $S^1$-equivariant quasi-equivalence
$$
\pi^*:\QCoh(\cH X) \risom \QCoh(\cH X_\bul)
$$
preserving subcategories of coherent objects.
\end{prop}

\begin{proof}
By the descent description of the algebra $\cO_{\cH X}$ provided by
Proposition~\ref{hochschild is completed odd tangent} and the discussion thereafter,
the
dg derived category $\QCoh(\cH X)$ can be calculated as the limit
dg derived category
$\QCoh(\cH X_\bul)$.
\end{proof}

Let us continue to work with
a smooth simplicial presentation $X_\bul\to X$ such that each of the simplices $X_k$ is a
smooth affine scheme.

On each simplex $X_k$,
consider the formal dg algebra of differential forms $\Omega^{-\bullet}_{X_k}$
placed in negative degrees, and
the formal dg algebra $\Omega^{-\bul}_{X_k}[d]$ obtained by adjoining
the de Rham differential as an element of degree $-1$.
By an $\Omega^{-\bullet}_{X_k}$-module (respectively, $\Omega^{-\bul}_{X_k}[d]$-module),
we will mean a dg $\Omega_{X_k}^{-\bul}$-module
(respectively, $\Omega^{-\bul}_{X_k}[d]$-module)
that is quasicoherent as an $\cO_{X_k}$-module.

We take
the dg derived categories $\QCoh(X,\Omega^{-\bullet}_X)$
and $\QCoh(X,\Omega^{-\bul}_{X}[d])$
to be the limits of the corresponding cosimplicial
dg derived categories.
It is not difficult to check that the limit dg categories are independent
of the choice of simplicial presentation $X_\bul\to X$.

\medskip

\begin{thm}\label{cyclic and D}
For a smooth Artin stack $X$, there are canonical quasiequivalences of dg
derived categories
$$
\QCoh(\cH X) \simeq \Omega_X^{-\bullet}-\module
$$
$$
\QCoh(\cH X)^{S^1} \simeq \Omega_X^{-\bullet}[d]-\module
$$
preserving subcategories of coherent objects.
\end{thm}

\begin{proof}
Observe that given a a simplicial presentation $X_\bul\to X$ such that each of the simplices $X_k$ is a
smooth affine scheme, all of
the dg categories under consideration are calculated by taking the
limit of the corresponding cosimplicial dg categories. Thus
it suffices to prove the theorem for $X$ a smooth affine scheme.

For $X$ a smooth affine scheme,
the first assertion is immediate from Proposition~\ref{hochschild is completed odd tangent}.
This is simply a reformulation of the Hochschild-Kostant-Rosenberg theorem
providing a quasi-isomorphism between Hochschild chains $C^{-\bullet}(\cO_X)$ and
the de Rham algebra $ \Omega_X^{-\bullet}.$

For the second assertion, the $S^1$-action on $\cH X$ corresponds to a cyclic
structure on $\cO_{\cH X}$. Under the Dold-Kan correspondence, the
cyclic structure goes over to the $\bLa$-structure on
$C^{-\bullet}(\cO_X)$ where the generator $\lambda\in \bLa^{-1}$
acts by the Connes (homological) differential. Then under the
Hochschild-Kostant-Rosenberg theorem, the Connes differential
becomes identified with the de Rham differential on $ \Omega_X^{-\bullet}$ (see \cite{Loday}).
Finally, Corollary \ref{cyclic to Lie} further identifies cyclic
modules over $C^{-\bullet}(\cO_X)$ with differential graded modules over the
dg algebra $\Omega_X^{-\bullet}[d]$. This establishes the second
assertion.
\end{proof}




\subsection{Koszul dual description}\label{D-modules}
Let $X$ be a smooth Artin stack with affine diagonal.
In this section, we explain the close connection between $S^1$-equivariant quasicoherent
sheaves on the Hochschild space $\cH X$ and filtered $\D$-modules on
$X$.
To make this precise, we will apply the Koszul duality functor of
Goresky-Kottwitz-MacPherson~\cite{GKM}.
This functor only differs from
 the Koszul duality functor of Kapranov~\cite{Kap dR}
and Beilinson-Drinfeld~\cite{BD Hitchin} with respect to its grading
convention.

\medskip

Choose a smooth simplicial presentation $X_\bul\to X$ such that each of the simplices $X_k$ is a
smooth affine scheme.

On each simplex $X_k$, consider the algebra of differential operators $\D_{X_k}$,
and for each $i\geq 0$, the subsheaf $\D_{X_k,\leq i}\subset \D_{X_k}$ of differential operators of order at most $i$.
Define the shifted Rees algebra $\cR_{X_k}$ to be the graded algebra
$$
\cR_{X_k}=\bigoplus_{i\geq 0} \D_{X_k,\leq i}[-2i]
$$
where the graded piece $\D_{X,\leq i}$ is placed in cohomological
degree $2i$.

Recall that $\Omega_X^{-\bullet}[d]$ denotes the de Rham algebra in
negative degrees with the de Rham differential $d$ adjoined as an
element of degree $-1$. If we think of $\cR_{X_k}$ as a left
$\cR_{X_k}$-module, then it admits a graded Koszul-de Rham complex
$\cK_{X_k}$. By definition, this is the dg $\cO_{X_k}$-module
$$
\cK_{X_k} =
(\Omega^{-\bul}_{X_k}[d]\otimes_{\cO_{X}} \cR_{X_k},\delta)
$$
where
the differential $\delta$ encodes the left action of $\cR_{X_k}$ on itself.
We will think of $\cK_{X_k}$ as
equipped with the obvious left action of $\Omega^{-\bul}_{X_k}[d]$ and right action of $\cR_{X_k}$.

On dg $\Omega^{-\bul}_{X_k}[d]$-modules $M$, and complexes of $\cR_{X_k}$-modules $N$,
we have adjoint functors
$$
F_{X_k}:N\mapsto \Hom_{\cR_{X_k}}(\cK_{X_k}, N)[-\dim X_k/X]
\qquad
G_{X_k}:M\mapsto M\otimes_{\Omega^{-\bul}_{X_k}[d]} \cK_{X_k}[\dim X_k/X]
$$

Consider the dg derived categories
$\Coh(X_k,\Omega^{-\bul}_{X_k}[d])$ and
$\Coh(X_k,\cR_{X_k})$
of dg modules whose cohomology is finitely generated over
the formal algebras
$\Omega^{-\bul}_{X_k}[d]$ and
$\cR_{X_k}$ respectively.
If we restrict to such modules,
then the above functors 
descend to quasi-equivalences
$$
F_{X_k}:\Coh(X_k,\cR_{X_k})\risom \Coh(X_k,\Omega^{-\bul}_{X_k}[d])
\qquad
G_{X_k}: \Coh(X_k,\Omega^{-\bul}_{X_k}[d])\risom\Coh(X_k,\cR_{X_k})
$$
For more details, the reader could consult Kapranov~\cite{Kap dR}.

Finally,
define the dg derived categories $\Coh(X,\Omega^{-\bul}_{X_\bul}[d])$ and
 $\Coh(X,\cR_X)$ to be the limits
of the corresponding cosimplicial dg derived categories.
Applying the above functors on each of the simplices
leads to analogous functors $F_{X}$ and $G_{X}$ on the limit dg categories.
For details of the following assertion, the reader could consult~\cite [Section 7.5]{BD Hitchin}.

\begin{thm}\label{Koszul de Rham}
The functors $F_X$ and $G_X$ provide quasi-equivalences of dg derived categories
$$
F_{X}:\Coh(X,\cR_{X})\risom \Coh(X,\Omega^{-\bul}_{X}[d])
\qquad
G_{X}: \Coh(X,\Omega^{-\bul}_{X}[d])\risom\Coh(X,\cR_{X})
$$
%
\end{thm}


\nc\bL{\mathbf L}


Putting together Theorems~\ref{cyclic and D}
and~\ref{Koszul de
Rham}, we immediately obtain the following.

\begin{corollary}\label{cyclic and rees}
There is a canonical quasi-equivalence
of dg derived categories
$$
\Coh(\cH X)^{S^1}\simeq \Coh(X,\cR_X).
$$
\end{corollary}


\subsection{Periodic sheaves and $\D$-modules}

Finally, we will consider the periodic, or localized, version of the
above picture. The category $\Coh(\cH_X/S^1)$ (much as every
equivariant category) is linear over the equivariant cohomology ring
of a point
$$
\bS=H^*(BS^1) = k[u] \qquad \mbox{ with $|u|=2$},
$$
the Koszul dual of the homology ring $\bLa$ of $S^1$.

For a $\bS$-linear dg category $\cC$ we may consider the
corresponding periodic category, in which we localize $\cC$ with
respect to the action of the central element $u$ (following
\cite{Toen dg}):
$$
\cC^{S^1}_{per}=\cC^{S^1}\ot_{k[u]}k[u,u\inv].
$$
Since $u$ has cohomological degree $2$, the morphism complexes in
the new category are $2$-periodic.

One may also periodicize a dg category by force: for a dg category
$\cC$ linear over $k$, we can always extend scalars to obtain a new
category
$$
\cC \otimes_k k[u,u\inv],
$$
where again the morphism complexes are $2$-periodic. (More
precisely, we take the pre-triangulated envelope of the naive tensor
product.)

\begin{corollary}
There is a canonical quasi-equivalence
of periodic dg categories
$$
\Coh(\cH X)^{S^1}_{per}\simeq \Coh(X,\D_X)\ot_k k[u,u\inv]
$$

\end{corollary}

\begin{proof}

Consider the central element $t\in \cR_X^2$ given by the unit of
$\cO_X\subset\D_{X,\leq 1}$ under the identification
$$
\D_{X,\leq 1}[-2]=\cR_X^2.
$$ We consider $\cR_X$ as an $\bS$-algebra, where $u\in \bS$ acts by
multiplication by $t$. Inverting $t$ in the graded algebra $\cR_X$
we obtain the periodic version of the algebra $\D_X$:
$$
\cR_X[t\inv]\simeq \D_X\ot_k k[t,t\inv].
$$
Inverting $t$ on the level of module categories, we obtain an
equivalence
$$
\Coh(X,\cR_X)\ot_{k[t]}k[t,t\inv]\simeq \Coh(X,\D_X)\ot_k
k[t,t\inv].
$$
(The equivalence follows from the existence of good filtrations on
coherent $\D$-modules.)

In order to conclude, we need to identify the action of $u\in
H^2(BS^1)$ on the $S^1$-equivariant category $\QCoh(\cH X)^{S^1}$
with the action of the central element $t\in \cR_X^2$ on Rees
modules under the quasi-equivalence of Corollary~\ref{cyclic and
rees}. This identification is a consequence of the evident
compatibility between the Koszul duality between $\bLa$ and $\bS$
and the Koszul duality between
$\Omega_X^{-\bul}[d]=\Oo_{\cH_X}[\bLa]$ and the $\bS$-algebra
$\bS\hookrightarrow \cR_X$.

\end{proof}

\begin{remark}
To obtain a $\Z$-graded version of the above corollary, we should work in a ``mixed"
setting with an extra grading direction.
\end{remark}

\section{Steinberg varieties as loop spaces}\label{Steinberg section}

\nc{\ch}{\check}
\nc{\fb}{\mathfrak b} \nc{\fu}{\mathfrak u} \nc{\st}{{st}}
\nc{\fU}{\mathfrak U}

\nc\fk{\mathfrak k} \nc\fp{\mathfrak p}

In this section, we apply the preceding results to a concrete
example in representation theory. We show how to relate two
prominently appearing categories associated to a complex reductive
group $G$. The first is the category of equivariant coherent sheaves
on the Steinberg variety of $G$.
The second is the category of Harish-Chandra bimodules of $G$ with
strictly trivial infinitesimal character.
Via Beilinson-Bernstein localization, we can identify the latter
with equivariant $\D$-modules on the flag variety of $G$. We refer
the interested reader to the overview in Section~\ref{overview} for
a brief discussion of the importance of these categories in
representation theory, and in particular for their interpretation as
Langlands parameters for representations of the Langlands dual
group. Our aim here is to see how to recover the second category
from the first via the geometry of loop spaces and $S^1$-equivariant
localization.

One of the complications in explaining this picture is the fact that
there is not one result but rather a family of results parametrized
by a product of universal Cartan groups $H\times H$. (These
parameters correspond to the infinitesimal character of
representations of the Langlands dual group.) To deal with this, we
introduce more and more local notions of what one might mean by the
Steinberg variety. First and foremost, there is the global Steinberg
variety $\St$ whose quotient $\St/G$ admits a realization as a loop
space. Then for each (interesting) value of the parameter $(\alpha,
w\alpha)$, we have the local Steinberg variety
$\St_{\alpha,w\alpha}$ obtained by completing the global Steinberg
variety $\St$ with respect to a neighborhood of the parameter
$(\alpha, w\alpha)$. Finally, we have the formal Steinberg variety
$\wh\St_{\alpha,\beta}$ obtained by completing the local Steinberg
variety $\St_{\alpha,w\alpha}$ with respect to its unipotent
directions.

It is the formal Steinberg space $\wh\St_{\alpha,\beta}/G$ that
turns out to be a Hochschild space, while the local Steinberg space
$\St_{\alpha,\beta}/G$ is what appears most commonly in
representation theory. The general results of previous sections go
so far as to relate equivariant $\D$-modules on certain flag
varieties to coherent sheaves on the equivariant formal Steinberg
space $\wh\St_{\alpha,\beta}/G$. To complete the bridge to
representation theory, we check that the restriction of coherent
sheaves from the local space $\St_{\alpha,\beta}/G$ to the formal
space $\wh\St_{\alpha,\beta}/G$ is an equivalence.

In the first two sections that follow, we explain the above story
for the trivial parameter. In the third and final section, we
explain the natural generalization for arbitrary parameters.

\medskip

We first set some notation. Let $ G$ be a connected reductive
complex algebraic group with Lie algebra $ \fg$, and let $ \cB$ be
the flag variety of $ G$ parameterizing Borel subgroups $ B\subset
G$. For each $ B\in  \cB$, we have the Cartan quotient $ H= B/ U$
where $ U\subset  B$ is the unipotent radical. The natural $
G$-action on $\cB$ by conjugation canonically identifies the Cartan
quotients for different $ B$, and so we call $ H$ the universal
Cartan of $ G$.


\subsection{The Steinberg space}
All of the spaces we will consider naturally live over the
classifying space $pt/ G$. (To avoid confusion with our notation for
Borel subgroups, we will denote the classifying space by $pt/G$
instead of the customary $BG$.) What follows is a brief introduction
to our {\em dramatis personae}.

\subsubsection{The adjoint group}
Let $ G/{G}$ be the adjoint group defined by taking the quotient of
$ G$ by the adjoint action of $ G$ on itself. As we have seen, as a
group space over $pt/G$, it is isomorphic to the loop space $\cL
(pt/G)$.

\subsubsection{The Grothendieck-Springer space}
Let $\wt G$ be the the Grothendieck-Springer variety of pairs of an
element $g\in G$ and a Borel subgroup $B \in \cB$ such that $g\in
B$. We refer to the quotient $\wt G/G$ by the adjoint action of $G$
as the Grothendieck-Springer space or equivariant
Grothendieck-Springer variety.

Since $G$ acts transitively on $\cB$ and the stabilizer of $B\in\cB$
is precisely $B$, we see that $\wt G/G$ is canonically isomorphic to
the adjoint group $B/{B}$, for any $B\in \cB$. Thus it is isomorphic
to the loop space $\cL (pt/B)$ as a group space over $pt/B$.

\medskip

Let $p:\wt G/G\to G/{G}$ be the projection $p(g,B)=g$. If we think
of $\wt G/G$ as the  loop space $\cL (pt/B)$, and $G/{G}$ as the
loop space $\cL (pt/G)$, then $p$ is the map on loops induced by the
natural fibration $pt/B\to pt/G$.

\medskip

Let $\pi:\wt G/G\to \cB/G$ be the other projection $\pi(g,B)=B$.
Again, if we think of $\wt G/G$ as the  loop space $\cL (pt/B)$, and
$ \cB/G$ as the classifying space $pt/B$, then $\pi$ is the natural
projection $\cL (pt/B) \to pt/B$ obtained by evaluating a loop at
$1\in S^1$.

 \subsubsection{The flag space}
Consider the diagonal $ G$-action on the product $ \cB\times \cB$ of
two copies of the flag variety. We refer to the corresponding
quotient stack $ (\cB \times\cB)/ G $ as the flag space or
equivariant flag variety. For a fixed element $(B_1,B_2)\in
\cB\times\cB$, we have a canonical identification
$$
B_1 \backslash G/ B_2 \risom ( \cB \times\cB)/ G \qquad g\mapsto
(B_1, g B_2 g^{-1}).
$$

 \subsubsection{The global Steinberg space}
 By the global Steinberg variety, we mean the fiber product
 $$
 \St = \wt G \times_G \wt G.
 $$
By a simple dimension count,
the above ordinary fiber product coincides with its derived
enhancement. We refer to the corresponding quotient stack
$$
\St/G = (\wt G \times_G \wt G)/G \simeq (\wt G/G \times_{G/G} \wt
G/G)
$$
as the global Steinberg space or equivariant global Steinberg
variety.

\medskip

The natural projection $\pi:\St/G\to (\cB \times\cB)/G$ realizes
$\St/G$ as a group-space over $(\cB \times\cB)/G$. Concretely, we
have two universal Borel subgroups
$$
B_1^{univ}, B_2^{univ}\subset (\cB \times\cB \times G)/G,
$$
and $\St/G$ is their intersection
$$
\St/G = B^{univ}_1\cap B^{univ}_2.
$$
Again thanks to a dimension count, the above ordinary intersection
coincides with its derived enhancement.

\medskip

For our purposes, the fundamental viewpoint on the global Steinberg
space is given by its natural realization as the loop space of the
flag space.

\begin{thm}\label{steinberg is a loop space}
We have a canonical identification
$$
 \St/G\simeq\cL  ((\cB \times\cB)/G)\simeq \cL(B\backslash G/B)
$$
of group spaces over $(\cB \times\cB)/G\simeq B\backslash G/B$.
\end{thm}

\begin{proof}
Recall that by definition, the inertia stack $I((\cB \times\cB)/G)$
is the underived mapping stack $\Hom(S^1, (\cB \times\cB)/G)$. It is
immediate from the definitions that $I((\cB \times\cB)/G)$ is
precisely the global Steinberg space $\St/G$. Thus to establish the
theorem, we must see that the loop space $\cL  ((\cB \times\cB)/G)$
coincides with $I ((\cB \times\cB)/G)$. In other words, we must see
that the derived structure of $\cL  ((\cB \times\cB)/G)$ is trivial.
But this is immediate from Proposition~\ref{dg structure of loops}.
\end{proof}


\subsubsection{The Steinberg space}

Let $ p:\St{\to}  H\times H $ be the natural projection to the
product of universal Cartans.
We call $H\times H$ the monodromic base, and refer to $p$ as the
monodromic projection.

Let $ \wh H_e \times\wh H_e$ be the formal neighborhood of the unit
$(e,e)\in H\times H$, and define the Steinberg variety $\St_{e,e}$
to be the inverse image
$$
\St_{e,e}=p^{-1}( \wh H_e \times\wh H_e).
$$
We refer to the corresponding quotient stack $ \St_{e,e}/G $ as the
Steinberg space or equivariant Steinberg variety.

Since $p$ is a group homomorphism, $\St_{e,e}/G$ is a group space
over $(\cB\times\cB)/G$.

\begin{remark}
It is worth pointing out that each component of the monodromic
projection $p$ more naturally takes values in the adjoint group
$H/H$. Of course, the conjugation action here is trivial, thus we
have $H/H \simeq H\times pt/H$, and so we were able to project to
the factor $H$. It is likely that the extra parameters coming from
the factor $pt/H$ will contribute meaningful deformations of the
current story.
\end{remark}


 \subsubsection{The formal Steinberg space}
We define the formal Steinberg variety $\wh \St_{e,e}$ to be the
formal completion of the Steinberg variety $\St_{e,e}$ along the
trivial section $(\cB\times\cB) \to \St_{e,e}$, or equivalently, the
formal completion of the global Steinberg variety $\St$ along the
trivial section $(\cB\times\cB) \to \St$. We refer to the
corresponding quotient stack $ \wh\St_{e,e}/G $ as the formal
Steinberg space or equivariant formal Steinberg variety.

\medskip

We have a diagram of group-spaces
$$
\wh \St_{e,e}/G \to \St_{e,e}/G \to \St/G
$$
over $(\cB\times\cB)/G$. The first is the formal group, the second
is the relative completion in the directions of the monodromic base,
and the last is the global group.

\medskip

By construction, Theorem~\ref{steinberg is a loop space} immediately
implies the following.

\begin{corollary}\label{formal steinberg is a hochschild space}
We have a canonical identification
$$
 \widehat \St_{e,e}/G \simeq\cH ((\cB\times\cB)/G) \simeq\cH(B\backslash G/B)
$$
of formal groups over $(\cB\times\cB)/G \simeq  B\backslash G/B$.
\end{corollary}


\subsection{Harish-Chandra bimodules}

Let $\fU(\fg \times \fg)$ denote the universal enveloping algebra of
the Lie algebra $\fg\times\fg$.
Consider the Harish-Chandra pair $(\fU(\fg\times\fg),G)$, and the
corresponding category 
of Harish-Chandra bimodules of $G$ with trivial infinitesimal
character. By definition, these are modules over
$(\fU(\fg\times\fg),G)$ which are finitely generated over $\fU(\fg
\times \fg)$ and on which the center $\mathfrak
Z(\fg\times\fg)\subset \fU(\fg\times\fg)$ acts via the trivial
character.

\medskip

By Beilinson-Bernstein localization, the abelian category of Harish
Chandra bimodules of $G$ with (strictly) trivial infinitesimal
character is equivalent to the abelian category of $G$-equivariant
$\D$-modules on $\cB\times\cB$. One should beware that the naive
corresponding dg derived category of Harish-Chandra bimodules is
{\em not} equivalent to the natural $G$-equivariant dg derived
category of $\D$-modules on $\cB\times\cB$. It is the latter
category that plays a more prominent role in representation theory
(for example, by \cite{BGS} it is Koszul dual to the derived
category of representations with generalized trivial infinitesimal
character, and by the work of ~\cite{Soergel} it arises as
parameters for representations of the dual group), and thus we will
proceed with it in mind as our model for the dg derived category of
Harish-Chandra bimodules. Alternatively, one could take the
appearance of Harish-Chandra bimodules as purely motivational, and
understand that it is the $G$-equivariant dg derived category of
$\D$-modules on $\cB\times\cB$ that we are most interested in.

\medskip

We have seen in Corollary~\ref{formal steinberg is a hochschild
space} that the formal Steinberg space $ \widehat \St_{e,e}/G$ is
the Hochschild space of the flag space $(\cB\times\cB)/G$. Thus
applying our general results, we have a canonical quasi-equivalence
of dg derived categories
$$
\Coh(\wh \St_{e,e}/G)^{S^1} \risom
\Coh(\cR_{(\cB\times\cB)/G}).
$$
Since the formal Steinberg space $\wh \St_{e,e}/G$ is not a commonly
appearing object in representation theory, the above identification
is not completely satisfactory. But thanks to the following lemma,
we can go one step further and relate these categories to coherent
sheaves on the Steinberg space $\St_{e,e}/G$ itself. Note that the
following lemma is the one place where we need to work with coherent
sheaves rather than quasicoherent sheaves.

\begin{lemma}\label{restriction is an equivalence}
The natural restriction map
$$
\Coh(\St_{e,e}/G) \to \Coh(\wh \St_{e,e}/G)
$$
is an $S^1$-equivariant equivalence.
\end{lemma}

\begin{proof}
The map $\wh \St_{e,e}/G\to  \St_{e,e}/G$ is clearly
$S^1$-equivariant, and so the $S^1$-equivariance of the lemma is
immediate.

To see the restriction is an equivalence, fix a Borel subgroup $B\in
\cB$ with unipotent radical $U\subset B$, and consider the formal
completion $\wh B_U$ along $U$ and the formal group $\wh B$. By base
change, it suffices to show that the restriction functor
$$
\Coh(\wh B_U/B) \to \Coh(\wh B/B).
$$
is an equivalence. Observe that each of the above categories can be
thought of as a category of representations of $B$ equipped with
extra structure.

Fix a generic one-parameter subgroup $\G_m\subset B$ so that the
induced conjugation action of $\G_m$ on the unipotent radical $U$
contracts it to the identity $e\in U$. It is straightforward to
check that the above restriction functor is equal to the completion
functor with respect to the weights of the $\G_m$-action.
Conversely, we can define an inverse functor
$$
\Coh(\wh B/B) \to \Coh(\wh B_U/B)
$$
by taking the $\G_m$-finite vectors of any object.
\end{proof}

Putting together Corollary~\ref{formal steinberg is a hochschild
space}, Lemma~\ref{restriction is an equivalence} and our general
results, we arrive at our goal as summarized in the following
statements.

\begin{thm}\label{unlocalized}
There is a canonical quasi-equivalence of dg derived categories
$$
\Coh(\St_{e,e}/G)^{S^1} \risom
\Coh(\cR_{(\cB\times\cB)/G}).
$$
\end{thm}

\begin{corollary}\label{localized}
There is a canonical quasi-equivalence of dg derived categories
$$
\Coh(\St_{e,e}/G)_{per}^{S^1} \risom
\Coh(\D_{(\cB\times\cB)/G})\otimes_\C \C[u,u^{-1}].
$$
\end{corollary}


\subsection{General monodromicities}
In this section, we establish results for a general monodromic
parameter analogous to those of the previous sections.

Recall that we have the monodromic projection $p:\St\to H \times H$,
and that the Steinberg variety $\St_{e,e}$ is defined to be the
inverse image under $p$ of the formal group $\wh H_e\times\wh H_e$.
Given an arbitrary pair $(\alpha,\beta)\in H\times H$, we denote its
formal neighborhood by $\wh H_\alpha \times\wh H_\beta$, and define
the monodromic Steinberg variety $\St_{\alpha,\beta}$ to be the
inverse image
$$
\St_{\alpha,\beta}=p^{-1}( \wh H_\alpha \times\wh H_\beta).
$$
We refer to the corresponding quotient stack $ \St_{\alpha,\beta}/G
$ as the monodromic Steinberg space or equivariant monodromic
Steinberg variety.

Our aim here is to give a loop space interpretation of coherent
sheaves on $\St_{\alpha,\beta}/G$.

\medskip

Let $W$ be the Weyl group of $G$.

\begin{lemma}
The monodromic Steinberg space $\St_{\alpha,\beta}/G$ is nonempty if
and only if there is $w\in W$ such that $\beta = w\alpha$.
\end{lemma}

\begin{proof}
For $(g,B_1,B_2)\in \St$, let $\alpha$ be the class of $g$ in
$B_1/U_1$ and $\beta$ the class of $g$ in $B_2/U_2$. We may
conjugate $B_1$ to $B_2$ to see that $\alpha$ must be conjugate to
$\beta$.
\end{proof}

Now fix $\alpha\in H$, and consider the monodromic Steinberg variety
$\St_{\alpha,w\alpha}$, for some $w\in W$.

\medskip

Let $\cO_\alpha\subset G$ denote the semisimple conjugacy class
corresponding to $\alpha$. Fix once and for all an element $\tilde
\alpha\in\cO_\alpha$, and let $G(\alpha)\subset G$ denote its
centralizer. In general, $G(\alpha)$ is reductive, and often turns
out to be a Levi subgroup.

We will affix the symbol $(\alpha)$ to our usual notation when
referring to objects associated to $G(\alpha)$. So for example, we
write $\cB(\alpha)$ for the flag variety of $G(\alpha)$, and
$\St(\alpha)$ for its global Steinberg variety. Furthermore, we have
the monodromic projection $p(\alpha):\St(\alpha)\to H \times H$, the
monodromic Steinberg variety $\St(\alpha)_{\alpha,\alpha}$ given by
the inverse image
$$
\St(\alpha)_{\alpha,\alpha}=p(\alpha)^{-1}( \wh H_\alpha \times\wh
H_\alpha),
$$
and the corresponding monodromic Steinberg space
$\St(\alpha)_{\alpha,\alpha}/G(\alpha)$.

\begin{thm}\label{monodromic reduction}
For each $w\in W$, we have a canonical $S^1$-equivariant
identification of monodromic Steinberg spaces
$$
 \St_{\alpha,w\alpha}/G\simeq \St(\alpha)_{\alpha,\alpha}/G(\alpha).
$$
\end{thm}

\begin{proof}
For each $w\in W$, one can check that there is a map
$$
\St(\alpha)_{\alpha,\alpha}\to \St_{\alpha,w\alpha}
$$
$$
(g, B(\alpha)_1, B(\alpha)_2) \mapsto (g, B_1, B_2)
$$
uniquely characterized by the properties:
$$
B_1\cap G(\alpha) = B(\alpha)_1 \quad B_2\cap G(\alpha) =
B(\alpha)_2
$$
$$
[g]_1 \in \wh H_\alpha \quad [g]_2 \in \wh H_{w\alpha}
$$
where $[g]_1,[g]_2$ denote the classes of $g$ in $B_1/U_1$,
$B_2/U_2$ respectively.

Passing to the respective quotients gives the sought-after
isomorphism.
\end{proof}

By Theorem~\ref{monodromic reduction}, to understand
$S^1$-equivariant coherent sheaves on $ \St_{\alpha,w\alpha}/G$ it
suffices to understand them on
$\St(\alpha)_{\alpha,\alpha}/G(\alpha)$. This is very close to a
problem we have already solved. Namely, consider the Steinberg space
$\St(\alpha)_{e,e}/G(\alpha)$ still for the group $G(\alpha)$, but
now for the trivial monodromic paramer. Applying
Theorem~\ref{unlocalized}, we obtain a canonical quasi-equivalence
of dg derived categories
$$
\Coh(\St(\alpha)_{e,e}/G(\alpha))^{S^1} \risom
\Coh(\cR_{(\cB(\alpha)\times\cB(\alpha))/G(\alpha)}).
$$
Moreover, multiplication by the fixed central element
$\tilde\alpha\in G(\alpha)$ provides an isomorphism
$$
\St(\alpha)_{e,e}/G(\alpha) \risom
\St(\alpha)_{\alpha,\alpha}/G(\alpha)
$$
$$
(g, B(\alpha)_1, B(\alpha)_2) \mapsto (\tilde\alpha g, B(\alpha)_1,
B(\alpha)_2).
$$
Thus putting the two statements together, we should be able to
conclude something about $S^1$-equivariant coherent sheaves on
$\St(\alpha)_{\alpha,\alpha}/G(\alpha)$.

Unfortunately, the above isomorphism is {\em not} $S^1$-equivariant,
and in general $S^1$-equivariant coherent sheaves on the two sides
will {\em not} be the same. What is true is that it induces a
quasi-equivalence on the level of {\em localized} $S^1$-equivariant
categories. Applying Corollary~\ref{localized}, we have a canonical
quasi-equivalence of dg derived categories
$$
\Coh(\St(\alpha)_{e,e}/G(\alpha))_{per}^{S^1} \risom
\Coh(\D_{(\cB(\alpha)\times\cB(\alpha))/G(\alpha)})\otimes_\C
\C[u,u^{-1}].
$$
Thanks to the rigidity of $\D$-modules (as opposed to Rees modules),
we have the following.

\begin{thm}\label{localized monodromic reduction}
There is a canonical quasi-equivalence of dg derived categories
$$
\Coh(\St(\alpha)_{e,e}/G(\alpha))^{S^1}_{per} \risom
\Coh(\St(\alpha)_{\alpha,\alpha}/G(\alpha))^{S^1}_{per}
$$
\end{thm}

\begin{proof}
Consider the isomorphism
$$
\St(\alpha)_{e,e}/G(\alpha) \risom
\St(\alpha)_{\alpha,\alpha}/G(\alpha)
$$
given by multiplication by the fixed central element
$\tilde\alpha\in G(\alpha)$. Let us compare the universal
monodromies of the respective $S^1$-actions under the induced
quasi-equivalence
$$
\Coh(\St(\alpha)_{e,e}/G(\alpha))\risom
\Coh(\St(\alpha)_{\alpha,\alpha}/G(\alpha)).
$$
By construction, we have an identity
$$
m_\alpha = m_e \circ \varphi_{\tilde\alpha}
$$
where $m_\alpha$ is the universal monodromy of the right hand side,
$m_e$ is that of the left hand side, and $\varphi_{\tilde \alpha}$
is the automorphism of the identity functor given by the central
element $\tilde \alpha \in G(\alpha)$.

Since automorphisms of the identity functor mutually commute, the
automorphism $\varphi_{\tilde \alpha}$ passes to the
$S^1$-equivariant category with respect to the monodromy $m_e$ and
further to its localization. We claim that after taking the
localized $S^1$-equivariant category with respect to the monodromy
$m_e$, the automorphism $\varphi_{\tilde \alpha}$ acts {\em
trivially}. If so, then every localized $S^1$-equivariant object or
morphism with respect to $m_e$ is canonically a localized
$S^1$-equivariant object or morphism with respect to $m_\alpha$ and
vice versa. Thus to prove the theorem, it suffices to establish the
above claim.

To prove the claim, let us consider the automorphism
$\varphi_{\tilde \alpha}$ under the identification
$$
\Coh(\St(\alpha)_{e,e}/G(\alpha))_{per}^{S^1} \risom
\Coh(\D_{(\cB(\alpha)\times\cB(\alpha))/G(\alpha)})\otimes_\C
\C[u,u^{-1}].
$$
Given any object of the stack
$(\cB(\alpha)\times\cB(\alpha))/G(\alpha)$, its automorphism group
contains a maximal torus $T(\alpha)\subset G(\alpha)$. Thus the
central element $\tilde\alpha\in G(\alpha)$ can be connected to the
identity by a path in the automorphism group. Thus it acts trivially
on any $\D$-module on $(\cB(\alpha)\times\cB(\alpha))/G(\alpha)$,
and the claim follows.
\end{proof}

Now returning to our original problem, Theorems~\ref{monodromic
reduction} and~\ref{localized monodromic reduction} and
Corollary~\ref{localized} immediately imply the following.

\begin{corollary}
For each $w\in W$, there is a canonical quasi-equivalence of
dg derived categories
$$
\Coh(\St_{\alpha,w\alpha}/G)_{per}^{S^1} \risom
\Coh(\D_{(\cB(\alpha)\times\cB(\alpha))/G(\alpha)})\otimes_\C
\C[u,u^{-1}].
$$
\end{corollary}




\nc{\RP}{\mathbf{RP}} \nc{\rigid}{\text{rigid}}
\nc{\glob}{\text{glob}}

\section{Langlands parameter spaces}\label{Langlands section}

This section contains our main intended application of the relation
between localized $S^1$-equivariant coherent sheaves on Hochschild
spaces and $\D$-modules. Given an involution $\eta$ of the complex
reductive group $G$, we introduce a stack $\St^\eta/G$ which we call
the global Langlands parameter space. Our aim in this section is to
explain the relationship between localized $S^1$-equivariant
coherent sheaves on $\St^\eta/G$ and $\D$-modules on certain flag
spaces. For further motivation, we refer the interested reader to
the overview and applications described in Sections~\ref{overview}
and \ref{applications}. Roughly speaking, the global Langlands space
$\St^\eta/G$ plays an analogous role to that of the global Steinberg
space $\St/G$ but now for {\em real forms} of the dual group
$G^\vee$. The flag spaces that arise in this section are precisely
the geometric parameter spaces of Adams, Barbasch and Vogan
\cite{ABV}, which appear in the geometry of Vogan duality and
Soergel's conjecture.

Many of the constructions and arguments of this section are direct
generalizations of those of Section~\ref{Steinberg section}. In
fact, the results of Section~\ref{Steinberg section} are the special
case when we take $G$ to be a product $G_o\times G_o$ and $\eta$ to
be the involution that switches the factors. We have chosen to
separate this case and explain it previously on its own for two
reasons. First, this case corresponds to the representation theory
of the complex dual group $G_o^\vee$ itself considered as a real
group, and so being of special significance warrants its own
statements.  Second, the proofs in the general case are more
complicated notationally but not conceptually.

In parallel with Section~\ref{Steinberg section}, the results of
this section will be parametrized by a (single) universal Cartan
group $H$. To deal with this, we introduce more and more local
notions of the global Langlands parameter space $\St^\eta/G$. For
each value of the parameter $\alpha$, we have the local Langlands
parameter space $\St^\eta_{\alpha}/G$ obtained by completing the
global Langlands parameter space $\St^\eta/G$ with respect to a
neighborhood of the parameter $\alpha$. Furthermore, we have the
formal Langlands parameter space $\wh\St^\eta_{\alpha}/G$ obtained
by completing the local Langlands parameter space
$\St^\eta_{\alpha}/G$ with respect to its unipotent directions.

In continued parallel with Section~\ref{Steinberg section}, it is
the formal Langlands parameter space $\wh\St^\eta_{\alpha}/G$ that
turns out to be a Hochschild space, while the local Langlands
parameter space $\St^\eta_{\alpha}/G$ is what plays a significant
role in representation theory as discussed in
Section~\ref{applications}. Our general results go so far as to
relate equivariant $\D$-modules on certain flag varieties to
coherent sheaves on the formal Langlands parameter space
$\wh\St^\eta_{\alpha}/G$. To complete the bridge to our desired
applications, we check that the restriction of coherent sheaves from
the local space $\St^\eta_{\alpha}/G$ to the formal space
$\wh\St^\eta_{\alpha}/G$ is an equivalence.

Because of the close parallels with Section~\ref{Steinberg section},
we have kept the arguments of this section brief; they often simply
refer to the analogous arguments of the preceding section. In the
first two sections that follow, we explain the story for the trivial
parameter. In the third and final section, we explain the natural
generalization for arbitrary parameters.

\medskip

We continue with the notation of the previous section so for
example, $G$ is a complex reductive group with flag variety $\cB$.
Fix once and for all an algebraic involution $\eta$ of $G$. Form the
semidirect product or $L$-group
$$
G_\eta = G \rtimes \Z/2\Z
$$
where the nontrivial element of $\Z/2\Z$ acts on $G$ by $\eta$. We
identify $G$ with the identity component of $G_\eta$, and write
$G_\eta \setminus G$ for the other component.


\subsection{The Langlands space}

Many of the spaces we will consider naturally live over the
classifying space $pt/G_\eta$. It sometimes will be useful to fix
the trivial $\Z/2\Z$-bundle on a point, and consider the resulting
base change diagram
$$
\begin{array}{ccc}
pt/G & \to & pt/G_\eta \\
\downarrow & & \downarrow \\
pt & \to & pt/ (\Z/2\Z)
\end{array}
$$
In general, for a space $Y$ living over $pt/G_\eta$, we will refer
to the base change
$$
Y_{\rigid}=Y\times_{pt/G_\eta} pt/G
$$
as the $\Z/2\Z$-rigidification of $Y$.

\medskip

As in the preceding section, we first introduce the {\em dramatis
personae} of our story.


\subsubsection{The adjoint group}
The loop space $\cL (pt/G_\eta)$ is nothing more than the adjoint
group $G_\eta/G_\eta$. Observe that $\cL (pt/G_\eta)$ has two
connected components
$$
\cL (pt/G_\eta) = G/ G_\eta \cup (G_\eta \setminus G)/G_\eta
$$
corresponding to $\Z/2\Z = \pi_0(G_\eta) \simeq \pi_1(pt/G_\eta)$.
The $\Z/2\Z$-rigidification of $\cL (pt/ G_\eta)$ is simply the
quotient $G_\eta/ G$ by the restriction of the conjugation action.


\subsubsection{The symmetric flag space}
Consider the diagonal $G$-action on the product $ \cB\times \cB$ of
two copies of the flag variety. We extend this to an action of
$G_\eta$ by setting
$$
(1, \eta)\cdot (B_1,B_2) = (\eta(B_2), \eta(B_1)).
$$
We refer to the corresponding quotient stack $ ( \cB \times\cB)/
G_\eta $ as the symmetric flag space.


\subsubsection{The global Langlands space}
Consider the loop space $\cL (( \cB \times\cB)/ G_\eta)$, and its
corresponding $\Z/2\Z$-rigidification $\cL (( \cB \times\cB)/
G_\eta)_{\rigid}$ Our immediate goal is to spell out what this space
is in terms of explicit equations.

\medskip

First and foremost, by Proposition~\ref{dg structure of loops}, we know that $\cL (( \cB
\times\cB)/ G_\eta)_{\rigid}$ is an ordinary stack with trivial
derived structure.

\medskip

Next, the canonical projection $( \cB \times\cB)/ G_\eta\to
pt/G_\eta$ induces a projection of rigidified loop spaces
$$
\cL (( \cB \times\cB)/ G_\eta)_\rigid\to G_\eta/G.
$$
Taking the preimages of the two connected components of $G_\eta/G$,
we obtain a decomposition of $\cL (( \cB \times\cB)/ G_\eta)_\rigid$
into two connected components.

The component of of $\cL (( \cB \times\cB)/ G_\eta)_\rigid$ above
the identity component of $G_\eta/G$ is canonically isomorphic to
the usual global Steinberg space $\St/G$. Both can be identified
with the loop space of the usual flag space $(\cB \times \cB)/G.$

\medskip

We write $\St^\eta/G$ for the second component of $\cL (( \cB
\times\cB)/ G_\eta)_\rigid$,
 and refer to it as the global Langlands space.
 By definition, it consists of paths in the flag space $( \cB \times\cB)/ G$ which begin
at a pair of flags $(B_1,B_2)$ and end at the pair
$(\eta(B_2),\eta(B_1))$.

To make this more explicit, consider the composite map
$$
\St^\eta/{G} \to  (\cB\times\cB)/G \to \cB/G
$$
given by projection to the first flag. Then the fiber of
$\St^\eta/{G}$ above a fixed flag $B\in \cB$ is the ordinary fiber
product
$$
B \times_G (G_\eta \setminus G)
$$
with respect to the inclusion $B\hookrightarrow G$ and the square
map $(G_\eta \setminus G)\to G$. Allowing the flag to vary, we see
that $\St^\eta/G$ is the moduli stack
$$
\St^\eta/G = \{ (B \in \cB, (g, \delta) \in B \times_G (G_\eta
\setminus G)\}/G.
$$
Simplifying the notation, we see that $\St^\eta/G$ is the stack
parametrizing pairs
$$
\St^\eta/G = \{B \in \cB, \delta\in (G_\eta\setminus G)
 \mbox{ such that }
 \delta^2 \in B\}/G.
$$

If we further identify $(G_\eta \setminus G)$ with $G$ itself, then
the square map becomes simply
$$g \mapsto \eta(g) g.
$$
Thus we may think of $\St^\eta_G$ as the stack parametrizing pairs
$$
\St^\eta/G= \{B \in \cB, g \in G
 \mbox{ such that }
  \eta(g) g \in B\}/G.
$$


 \subsubsection{The Langlands space}

Let $H=B/U$ be the universal Cartan, and consider the natural
projection $ p:\St^\eta/{G}{\to}  H $ given by
$$
p(B,\delta) = [\delta^2] \in B/U.
$$
We call $H$ the monodromic base, and refer to $p$ as the monodromic
projection.

Let $ \wh H_e$ be the formal neighborhood of the unit $e\in H$, and
define the Langlands space $\St^\eta_e/G$ to be the inverse image
$$
\St^\eta_e/G =p^{-1}( \wh H_e).
$$


 \subsubsection{The formal Langlands space}

\nc{\cI}{\mathcal I}

The Langlands space $\St_e^\eta/G$ breaks up into many connected
components.
To describe this decomposition, consider the space of involutions
$$
\cI = \{\iota \in (G^\eta\setminus G) | \iota^2 = 1\}.
$$
For convenience, fix once and for all an element $\iota$ in each
connected component of $\cI$, and let $\cI_\iota$ denote the
connected component containing $\iota$.

Each $\iota$ provides an involution of $G$ by conjugation, and we
write $K_\iota\subset G$ for the corresponding fixed point subgroup.
By construction, the quotient stack $\cI_\iota/G$ is isomorphic to
the classifying space $pt/K_\iota$.

For each $\iota$, we write $\cN_\iota$ for the connected component
of the image of the projection
$$
\St_e^\eta\to (G_\eta \setminus G).
$$
containing $\iota$. Then we have that $\St^\eta_e/G$ is the disjoint
union of the connected components
$$
\St^\iota_e/G= \{ (B \in \cB, (g, \delta) \in B \times_G \cN_\iota
\}/G.
$$
As in the construction of $\St^\eta/G$, the ordinary fiber product
here is with respect to the inclusion $B\hookrightarrow G$ and the
square map $\cN_\iota \to G$.

Consider the $K_\iota$-action on the flag variety $\cB$ and the
canonical embedding
$$
e: \cB/ K_\iota \to \St^\iota_e/G \qquad e(B) =  (B,(1,\iota)).
$$
We write $\wh\St^\iota_e/G$ for the formal neighborhood of the image
of $e$, and refer to it as the formal Langlands space for $\iota$.

\begin{thm}\label{formal langlands is a hochschild space}
There is a canonical isomorphism
$$
\cH (\cB/ K_\iota) \simeq \wh\St^\iota_e/G.
$$
\end{thm}

\begin{proof}
First, observe that the left hand side parametrizes pairs
$$
\cH (\cB/ K_\iota) =\{(B\in\cB, (g,k) \in B\times_G \wh
K_\iota)\}/K_{\iota}
$$
where $\wh K_\iota$ denotes the formal group of $K_\iota$. Here the
ordinary fiber product is with respect to the inclusions of
subgroups $B\hookrightarrow G$ and $\wh K_\iota \hookrightarrow G$.

Recall that the Langlands space $\St^\iota_e/G$ parametrizes pairs
$$
\St^\iota_e/G=\{(B\in\cB, (g,\delta) \in B\times_G \cN_\iota)\}/G
$$
where the ordinary fiber product is with respect to the inclusion
$B\hookrightarrow G$ and the square map $\cN_\iota \to G$.

Now we can define a map
$$
\cH (\cB/ K_\iota) \to \wh\St^\iota_e/G \qquad (B, (g,k) ) \mapsto
(B, (g^2, k\iota)).
$$
That this is an isomorphism follows from the constructions and the
fact that for any formal group in characteristic zero, the squaring
map is an isomorphism.
\end{proof}

As can be seen explicitly from the proof, Theorem~\ref{formal
langlands is a hochschild space} does {\em not} in general provide
an $S^1$-equivariant isomorphism. Let us compare the universal
monodromies of the respective $S^1$-actions under the induced
quasi-equivalence
$$
\QCoh(\cH (\cB/ K_\iota) )\risom \QCoh(\wh\St^\iota_e/G).
$$
By construction, we have an identity
$$
m^\iota_\St = m^\iota_\cH \circ \varphi_{\iota}
$$
where $m^\iota_{\St}$ is the universal monodromy of the right hand
side, $m^\iota_{\cH}$ is that of the left hand side, and
$\varphi_{\iota}$ is the automorphism of the identity functor given
by the element $\iota \in G_\eta \setminus G$.

Although we will not use it, it is also worth pointing out that
since $\iota^2=1$, the squared $S^1$-actions on each side in fact
coincide.


\subsection{Harish-Chandra modules}
Let $\fU(\fg)$ denote the universal enveloping algebra of the Lie
algebra $\fg$.
For a fixed involution $\iota$ of $G$ with fixed-point subgroup
$K_\iota$, consider the Harish-Chandra pair $(\fU(\fg), K_\iota)$,
and the corresponding category of Harish-Chandra modules with
trivial infinitesimal character. By definition, these are modules
over the Harish-Chandra pair $(\fU(\fg), K_\iota)$ which are
finitely generated over $\fU(\fg)$ and on which the center
$\mathfrak Z(\fg)\subset \fU(\fg)$ acts via the trivial character.

\medskip

By Beilinson-Bernstein localization, the abelian category of Harish
Chandra modules for $(\fU(\fg), K_\iota)$ with trivial infinitesimal
character is equivalent to the abelian category of
$K_\iota$-equivariant
$\D$-modules on $\cB$. One should beware that the naive
corresponding dg derived category of Harish-Chandra modules is {\em
not} equivalent to the natural $K_\iota$-equivariant dg derived
category of $\D$-modules on $\cB$. It is the latter category that
plays a more prominent role in representation theory (for example,
by the work of ~\cite{ABV,Soergel} as parameters for certain
representations), and thus we will proceed with it in mind as our
model for the dg derived category of Harish-Chandra modules.
Alternatively, one could take the appearance of Harish-Chandra
modules as purely motivational, and understand that it is the
$K_\iota$-equivariant dg derived category of $\D$-modules on $\cB$
that we are most interested in.

\medskip

Now
we arrive at our goal of identifying localized $S^1$-equivariant
coherent sheaves on the component of the Langlands space
$\St^\iota_e/G$.

\begin{thm}\label{localized langlands}
There is a canonical quasi-equivalence of dg derived categories
$$
\Coh(\St^\iota_e/G)_{per}^{S^1} \risom
\Coh(\D_{\cB/K_\iota})\otimes_\C \C[u,u^{-1}].
$$
\end{thm}

The proof is completely analogous to arguments appearing in
Section~\ref{Steinberg section}. First, as in Lemma~\ref{restriction
is an equivalence}, one checks that the natural restriction map
$$
\Coh(\St^\iota_{e}/G) \to \Coh(\wh \St^\iota_{e}/G)
$$
is an $S^1$-equivariant equivalence. Then as in
Theorem~\ref{localized monodromic reduction}, one uses
Theorem~\ref{formal langlands is a hochschild space} and the
discussion thereafter to deduce an equivalence of {localized}
$S^1$-equivariant categories. We leave it to the interested reader
to trace through the arguments.


\subsection{General monodromicities}
In this section, we establish results for a general monodromic
parameter analogous to those of the previous sections.

Recall that we have the monodromic projection $p:\St^\eta/{G}\to H$,
and that the Langlands space $\St^\eta/G$ is defined to be the
inverse image under $p$ of the formal group $\wh H_e$.

More generally, given $\alpha\in H$, we denote its formal
neighborhood by $\wh H_\alpha$, and define the monodromic Langlands
space $\St^\eta_{\alpha}/G$ to be the inverse image
$$
\St^\eta_{\alpha}/G=p^{-1}( \wh H_\alpha).
$$

Our aim here is to give a loop space interpretation of coherent
sheaves on $\St^\eta_{\alpha}/G$.

\medskip

As with the case already considered when $\alpha$ is the identity,
we have the following picture. The monodromic Langlands space
$\St^\eta_{\alpha}/G$ breaks up into many connected components. To
describe this decomposition, fix a semisimple representative $\tilde
\alpha \in G$, and consider the space of elements
$$
\cI_\alpha = \{\iota \in (G^\eta\setminus G) | \iota^2 =
\tilde\alpha\}.
$$
For convenience, fix once and for all an element $\iota$ in each
connected component of $\cI_\alpha$, and let $\cI_{\alpha, \iota}$
denote the connected component containing $\iota$.

Let $G(\alpha)\subset G$ be the centralizer of $\tilde \alpha$. Each
$\iota$ provides an involution of $G(\alpha)$ by conjugation, and we
write $K(\alpha)_\iota\subset G(\alpha)$ for the corresponding
fixed-point subgroup. By construction, the quotient stack
$\cI_{\alpha, \iota}/G(\alpha)$ is isomorphic to the classifying
space $pt/K(\alpha)_\iota$.

For each $\iota$, we write $\cN_{\alpha, \iota}$ for the connected
component of the image of the projection
$$
\St^\eta_\alpha\to (G_\eta \setminus G).
$$
containing $\iota$. Then we have that $\St_\alpha^\eta/G$ is the
disjoint union of the connected components
$$
\St_\alpha^\iota/G= \{ (B \in \cB, (g, \delta) \in B \times_G
\cN_{\alpha, \iota} \}/G.
$$
As in the construction of $\St^\eta/G$, the ordinary fiber product
here is with respect to the inclusion $B\hookrightarrow G$ and the
square map $\cN_{\alpha, \iota} \to G$.

Let $\cB(\alpha)$ be the flag variety of $G(\alpha)$, and consider
the $K(\alpha)_\iota$-action on $\cB(\alpha)$. Given a Borel
$B(\alpha)\in \cB(\alpha)$, there is a unique Borel $B\in\cB$
containing $B(\alpha)$ such that the class of $\tilde \alpha$ in the
universal Cartan $H=B/U$ is equal to $\alpha$. Thus we have a
canonical embedding
$$
e: \cB(\alpha)/ K(\alpha)_\iota \to \St^\iota_\alpha/G \qquad
e(B(\alpha)) =  (B,(\tilde\alpha,\iota))
$$
We write $\wh\St_\alpha^\iota/G$ for the formal neighborhood of the
image of $e$, and refer to it as the formal monodromic Langlands
space for $\iota$.

\medskip

The following is easily checked as in Theorem~\ref{formal langlands
is a hochschild space} when $\alpha$ is trivial. We leave further
details to the interested reader.

\begin{thm}\label{formal monodromic langlands is a hochschild space}
There is a canonical isomorphism
$$
\cH (\cB(\alpha)/ K(\alpha)_\iota) \simeq \wh\St_\alpha^\iota/G.
$$
\end{thm}

As in Theorem~\ref{formal langlands is a hochschild space}, the
identification of Theorem~\ref{formal monodromic langlands is a
hochschild space} does {\em not} in general provide an
$S^1$-equivariant isomorphism. Let us compare the universal
monodromies of the respective $S^1$-actions under the induced
quasi-equivalence
$$
\QCoh(\cH (\cB(\alpha)/ K(\alpha)_\iota) )\risom
\QCoh(\wh\St^\iota_\alpha/G).
$$
Under this identification, we have an identity
$$
m^\iota_{\St_\alpha} = m^\iota_{\cH_\alpha} \circ \varphi_{\iota}
$$
where $m^\iota_{\St_\alpha}$ is the universal monodromy of the right
hand side, $m^\iota_{\cH_\alpha}$ is that of the left hand side, and
$\varphi_{\iota}$ is the automorphism of the identity functor given
by the element $\iota \in G_\eta \setminus G$.

\medskip

We arrive at our goal of indentifying localized $S^1$-equivariant
coherent sheaves on the component of the mondromic  Langlands space
$\St^\iota_\alpha/G$.

\begin{thm}
There is a canonical quasi-equivalence of dg derived categories
$$
\Coh(\St^\iota_{\alpha}/G)_{per}^{S^1} \risom
\Coh(\D_{\cB_\alpha/K_{\alpha,\iota}})\otimes_\C \C[u,u^{-1}].
$$
\end{thm}

The proof is completely analogous to that of Theorem~\ref{localized
langlands} and follows arguments appearing in Section~\ref{Steinberg
section}. First, as in Lemma~\ref{restriction is an equivalence},
one checks that the natural restriction map
$$
\Coh(\St^\iota_{\alpha}/G) \to \Coh(\wh \St^\iota_{\alpha}/G)
$$
is an $S^1$-equivariant equivalence. Then as in
Theorem~\ref{localized monodromic reduction}, one uses
Theorem~\ref{formal monodromic langlands is a hochschild space} and
the discussion thereafter to deduce an equivalence of {localized}
$S^1$-equivariant categories. We leave it to the interested reader
to trace through the arguments.


\section{Appendix: derived stacks}\label{appendix}

In what follows, all schemes, stacks, etc are assumed to be over a field
$k$ {\em of characteristic $0$}. Unless otherwise stated, all rings are assumed to be commutative
with unit.



\subsection{Some motivation}
Our main objects of study are {\em derived stacks} in the sense of \cite{Toen,Lurie}.
Before recalling what is meant by a derived stack, we informally review some motivation for their introduction.

As we will outline momentarily, derived stacks are a broad context for dealing
with natural questions in algebraic geometry.
But perhaps the best motivation for considering derived stacks is not what they bring to algebraic
geometry, but that they bring algebraic geometry to other areas.
One of the most exciting examples of this is the unity of stable homotopy theory
and derived formal groups
in the language of the ``brave new algebraic geometry" of $E_\infty$-ring spectra.
In this paper, we discuss
how derived algebraic geometry provides a natural language for discussing basic objects of geometric representation theory. In particular,
Steinberg varieties and $\D$-modules
may be described
as derived stacks and quasicoherent sheaves on derived stacks respectively.

Our starting point is the study of schemes, or more generally Artin algebraic spaces.
(The latter are obtained from the former by allowing \'etale equivalence
relations.) Throughout the discussion, we will think of such geometric objects via
their functors of points. Thus by an algebraic space, we will mean
a functor
$$
\cF:\mbox{Rings} \to \Sets
$$
that is a sheaf in the \'etale topology, and admits an \'etale atlas
(a representable \'etale surjection
$
U\to \cF
$
where $U$ is a scheme; the \'etale equivalence relation on $U$ is given by the fiber
product $U\times_\cF U$).

Even if one is only interested in schemes or algebraic spaces,
natural geometric constructions produce
objects which lie beyond their definition. Two fundamental examples worth keeping in mind
are the problems of finding universal families and forming intersections. The difficulties
of the first are well-known: it is impossible to find a space
that parametrizes familiar objects
such as curves or vector bundles. In other words, there is an insurmountable obstruction to
constructing
a space whose points are in natural bijection with the objects under consideration.
The source of the difficulty is that such objects come with large amounts of symmetry; no space
will be rich enough to parametrize all possible twisted families.
To overcome this, there is the theory of algebraic stacks: one expands the notion
of a space to include functors
$$
\cF:\mbox{Rings} \to \mbox{Groupoids}.
$$
Natural notions of representability include Deligne-Mumford stacks (where $\cF$ is a sheaf
in the \'etale topology and admits an \'etale atlas)
and Artin stacks
(where $\cF$ is a sheaf
in the faithfully flat quasi-compact topology and admits a
smooth atlas).
Thus rather than trying to wrestle with the automorphisms of objects, we accept their
presence and parametrize the objects and their symmetries simultaneously.
It is worth commenting that if one continues with such considerations, one finds the same deficiency
in restricting to functors with values in groupoid. A hint of higher stacks appears:
one should rather consider
functors with values in arbitrary simplicial sets.

A similarly well-known difficulty is inherent in the intersection theory of schemes.
Basic facts about intersections of schemes fail as soon as the intersection is degenerate.
The usual solution is to recognize that the discrepancy may be accounted for by higher
homological invariants of the intersection. Rather than only considering
the naive tensor product which defines the scheme-theoretic intersection, we should also
keep track of the higher $\on{Tor}$ terms as well. Considering such derived functors
and their underlying differential graded avatars
has become central in the study of coherent sheaves.
But one aspect of the situation has only come into focus relatively recently:
when derived constructions (such as the tensor product of rings)
produce a differential graded ring, we should continue
to regard this as kind of generalized scheme. The resulting theory of derived schemes
considers functors from differential graded rings to simplicial sets; more generally (though equivalent to the differential graded theory in characteristic zero),
one studies functors of the form
$$
\cF:\mbox{Simplicial rings} \to \mbox{Simplicial sets}.
$$
Here the appearance of simplicial sets follows naturally from the fact that we have introduced
simplicial rings. For example, since morphisms between simplicial rings are enriched
over simplicial sets,
representable functors naturally take values in simplicial sets.
(But it is worth mentioning that the derived schemes that are the building blocks
of the theory to be discussed will continue to take ordinary rings
to ordinary sets.)
Although for the time being we are postponing any formal details, it is important to comment
that what we care about here is not the simplicial structure on our rings and sets, but rather only
homotopically meaningful properties. Thus via the geometric realization of simplicial sets, we may equivalently consider
functors on topological rings with values in topological spaces. It is worth mentioning that
one may also take the perspective that a derived
scheme is a kind of ``locally ringed space" with structure sheaf a topological ring.

Finally, we arrive at the theory of derived stacks by passing to the natural
level of generality
implicit in the above theories
of stacks and derived schemes. Namely, we continue to consider functors of the form
$$
\cF:\mbox{Simplicial rings} \to \mbox{Simplicial sets}
$$
but no longer make such strict assumptions about the vanishing of higher homotopy groups of
the functors when evaluated on ordinary rings.
As before, we only care about the homotopic properties of test objects and functor values, and
so could equivalently consider functors from topological rings to topological spaces.
In the next section, we will recall the setting of quasicategories which keeps track of
the correct amount of structure. It is worth mentioning one more motivation for
considering derived stack. If we have accepted that stacks
are fundamental objects, then we would hope our working framework would encompass
natural constructions involving them. As a basic example, one can see that
the cotangent bundle of something as simple as a classifying stack is already a derived stack.



\subsection{Basic terminology}

Our aim in this section is to review some of the basic terminology in the theory
of derived stacks. It is a formidable task to come to terms with the intricate foundation
of this theory. Fortunately,
there are several excellent sources to which we may refer the reader~\cite{Toen,Lurie}.
In the following, we
content ourselves with introducing the
main objects and their most relevant attributes.

One challenge of introducing derived stacks is their unfamiliar categorical underpinnings.
A natural setting for their discussion is not category theory
strictly speaking but higher category theory in the form of
topological categories. In other words, we should work in the
context of categories enriched over topological spaces: the
morphisms between objects are topological spaces and the composition
maps are continuous. While this is a correct approach (and
equivalent to the approach adopted below), it is often unnecessarily
restrictive in practice. To provide more elbow room, we should
rather work in some more flexible version of $(\infty,1)$-categories
such as quasicategories (there is also the alternative framework of
Segal categories \cite{HAG1}, see \cite{Bergner} for a comparison
between the different frameworks). To give the definition of
quasicategories is easy: they are certain simplicial sets sometimes
called weak Kan complexes. Let $\Delta^n$ denote the standard
$n$-simplex, and $\Lambda_i^n\hookrightarrow \Delta^n$ the ``inner
horn" obtained by deleting the interior open $n$-simplex and the
$i$th $(n-1)$-dimensional face.

\begin{defn}
A quasicategory is a simplicial set $K$ satisfying the following extension property:
for any $0<i<n$, any map $\Lambda_i^n \to K$ extends to a map $\Delta^n\to K$.
\end{defn}

The theory of quasicategories is well documented in the literature,
and there are many good sources for the interested reader (see in
particular \cite{Joyal} and the survey \cite{Bergner}, in addition
to the book \cite{topoi}). Rather than recalling any further formal
properties, it may be more meaningful to try to spell out what
motivates the definition. Roughly speaking, one wants the notion of
a category whose morphisms are topological spaces and whose
compositions and associativity properties are defined up to coherent
homotopies. For a quasicategory, the objects are given by its
$0$-simplices, and the $n$-morphisms by its $n$-simplices. Though
there is no explicit mention of an associative composition in the
above definition, the extension property is exactly what is needed
to define such structure up to coherent homotopies. Finally, given a
topological category, one may construct a quasicategory by taking
its topological nerve (which is by definition the simplicial nerve
of its singular complex). This sets up an equivalence between the
theory of topological categories and that of quasicategories.

A common way that quasicategories arise is
via the localization of simplicial model categories.
One may think of many quasicategories
as fitting into a sequence
$$
\mbox{Simplicial model category $\cC$}
\to\mbox{Quasicategory $N(\cC^\circ)$}
\to\mbox{Homotopy category $h\cC$}
$$
where the quasicategory $N(\cC^\circ)$ is the simplicial nerve of the subcategory
$\cC^\circ\subset \cC$ of fibrant-cofibrant objects.
More generally, one can associate to any model category an underlying quasicategory
by inverting the weak equivalences in the appropriate sense.
For example, starting from the categories of
(compactly generated Hausdorff) topological spaces
and simplicial sets with their usual model structures,
we obtain canonically equivalent quasicategories which we
will identify, denote by $\cS$, and refer to as the quasicategory of spaces.
Note that
passing to the homotopy category is often too drastic: there is not enough structure
in order to make usual constructions.
To make an analogy, we sometimes think of the following toy model of the above sequence
$$
\mbox{Based vector spaces}
\to\mbox{Vector spaces}
\to\mbox{Dimensions of vector spaces}.
$$
While it is often best to consider vector spaces with no preferred basis,
standard constructions can not be made at the level of their dimensions.

\medskip

With the preceding discussion in hand, we now proceed to recall the definition
of derived stacks. As informally discussed in the previous section, we will think of
derived stacks
in terms of their functors of points. Thus the first order of business is to describe
what our test objects are and where our functors take their values.

Let $\cC\cA_k$ denote the category of commutative
unital $k$-algebras, and let $\cC\cA_k^\Delta$ denote its simplicial category.
We endow the latter with the structure of a simplicial 
model category in which the weak equivalences and fibrations are
weak equivalences and fibrations on the underlying simplicial sets.
Let $\SCA_k$ denote the quasicategory obtained from $\cC\cA_k$ by
considering its fibrant-cofibrant objects. One refers to objects
of its opposite quasicategory as affine derived schemes.
It is possible to speak of \'etale and smooth morphisms between affine derived schemes,
and hence in particular the \'etale topology on $\SCA_k$.
Recall that $\cS$ denotes the quasicategory of spaces obtained from the
simplicial model category
of  simplicial sets, or equivalently (compactly generated Hausdorff) spaces, by considering
fibrant-cofibrant objects.

\begin{defn}
A {derived scheme} over $k$ is a functor
$$
\cF:\SCA_k\to \cS
$$
that is a sheaf in the \'etale topology,
and admits an \'etale atlas $U\to \cF$ where $U$ is an affine derived scheme.
\end{defn}

The reader will notice that the definition is more akin to that of a Deligne-Mumford stack
than an ordinary scheme. This turns out to be a more natural notion from the perspective
of locally ringed spaces, or more accurately ringed topos theory.
Namely, one may alternatively define a derived scheme to be
an $\infty$-topos $\cX$ equipped with a $\SCA_k$-valued structure sheaf $\cO$
satisfying the following representability:
there is a collection of objects $U_\alpha\in \cX$ such that the map
$\coprod_\alpha U_\alpha \to 1_\cX$ is surjective, and each $\infty$-topos
$\cX_{/U_\alpha}$ with structure sheaf $\cO_\cX|U_\alpha$
is equivalent to the spectrum of some simplicial commmutative ring.
The allowable gluings of the categorical setting of topoi naturally lead
to the generality of Deligne-Mumford stacks.
To arrive at the derived notion of algebraic space, we simply impose
the following condition on the gluings.


\begin{defn}
A {derived algebraic space} over $k$ is a derived scheme $\cF$
such that $\cF(A)$ is discrete whenever $A$ is discrete.
\end{defn}

Now to pass to arbitrary derived Artin stacks, we relax the representability assumption
of a derived algebraic space.
(Note that we are already considering functors with values in
the quasicategory of spaces $\cS$ so there
is no need to generalize anything in this direction.)
We will use the term derived stack to refer to any functor
$$
\cF:\SCA_k\to \cS
$$
that is a sheaf in the \'etale topology.
Let us consider what kind of representability we could
allow for an atlas
$p:U\to \cF$. Since all of our previous objects (specifically derived algebraic spaces)
admit atlases of affine derived schemes, we will keep to this here as well and assume
$U$ is an affine derived scheme.
But now it makes sense to allow $p$ to be a smooth relative derived algebraic space
rather than only an \'etale map.
One calls sheaves $\cF$ that admits such a presentation $1$-Artin stacks.
Now of course, we could
iterate this definition and allow \'etale sheaves $\cF$ that admit
atlases
$p:U\to \cF$ where $U$ is an affine derived scheme, but $p$
is a smooth relative $1$-Artin stack.
One calls sheaves $\cF$ that admits such a presentation $2$-Artin stacks.
And so on: the inductive story is encapsulated in the following definition.

\begin{defn}
Let $\dstack$ be the quasicategory of derived stacks whose objects are
$\cS$-valued \'etale sheaves on $\SCA_k$.
In what follows, let $X\to Y$ be a map of such sheaves,
and let $T$ be an arbitrary affine derived test scheme.

\smallskip

A morphism $p:X\to Y$ is a relative $0$-Artin stack if for any map $T\to Y$,
the base change $T\times_Y X$ is a derived algebraic space. Such a map
$p$ is said to be
smooth if the induced map $T\times_Y X\to T$ is smooth as a map of derived schemes.

\smallskip

For $n>0$, a morphism $p:X\to Y$ is a relative $n$-Artin stack if for any map $T\to Y$,
there exists an affine derived scheme $U$, and a smooth
surjection $U\to T\times_Y X$ which is a relative $(n-1)$-Artin stack. Such a map
$p$ is said to be
smooth if the induced map $U\to T$ is smooth.

\smallskip

A derived $n$-Artin stack $X$ is a relative derived $n$-Artin stack $X\to \Spec(k)$.
A derived Artin stack is a derived $n$-Artin stack for some $n\geq 0$.
\end{defn}

It is worth remarking that the basic building blocks of derived Artin stacks
(according to the above definition)
are {\em affine} derived schemes as opposed to all derived schemes.
In particular, not all derived schemes are derived stacks in the above sense.
Rather, one can check that for $n\geq 1$,
a derived scheme $\cF$ is a derived $n$-stack if and only if
for all discrete test rings $A$, the homotopy groups of $\cF(A)$ vanish
in degrees greater than or equal to $n$.


%




%

\medskip

One final issue to comment upon before wrapping up this survey is our choice of simplicial
commutative rings as coefficients. There are several possible generalizations
of the ordinary category of discrete rings which one might consider when
defining affine derived schemes and hence derived stacks. In the preceding
discussion, we have worked with
the quasicategory $\SCA_k$ of simplicial commutative $k$-algebras,
or equivalently topological $k$-algebras. To an algebraically minded person,
this world may feel very far from the intuitions of discrete rings. Thus it is
worth pointing out that in characteristic zero, in place of
topological $k$-algebras
one may instead consider the category of
commutative differential graded $k$-algebras.
It admits a model structure in which
the weak equivalences are simply the quasiisomorphisms,
and the cofibrations are retracts of iterated cell attachments.
We write $\DGA_k$ for the underlying quasicategory.
The Dold-Kan theorem provides a canonical
fully faithful embedding
$$
\SCA_k\hookrightarrow\DGA_k.
$$
Its essential image consists of connective objects (those objects whose cohomology
vanishes in negative degrees).
Thus in characteristic zero,
we may think of the quasicategory of affine derived schemes as
opposite to that of
connective differential graded $k$-algebras.
In the main text of this paper, we freely interpolate between
the simplicial and differential graded points of view.


\end{document}